\documentclass[11pt,reqno,tbtags]{amsart}

\usepackage{amsthm,amssymb,amsfonts,amsmath}
\usepackage{amscd} 
\usepackage{mathrsfs}
\usepackage[numbers, sort&compress]{natbib}
\usepackage[lofdepth,lotdepth]{subfig}
\usepackage{graphicx}
\usepackage{vmargin}
\usepackage{setspace}
\usepackage{paralist}
\usepackage[pagebackref=true]{hyperref}
\usepackage{stmaryrd} 
\usepackage{bm} 
\usepackage[refpage,noprefix,intoc]{nomencl} 

\providecommand{\R}{}
\providecommand{\Z}{}
\providecommand{\N}{}

\renewcommand{\R}{\mathbb{R}}
\renewcommand{\Z}{\mathbb{Z}}
\renewcommand{\N}{{\mathbb N}}

\newcommand{\p}[1]{{\mathbf P}\left\{#1\right\}}
\newcommand{\psub}[2]{{\mathbf P}_{#1}\left\{#2\right\}}

\newcommand{\I}[1]{{\mathbf 1}_{[#1]}}

\newcommand\cL{{\mathcal L}}

\newcommand\cT{{\mathcal T}}

\newcommand{\bP}{\mathbf{P}}

\newcommand{\convdist}{\ensuremath{\stackrel{\mathrm{d}}{\rightarrow}}}

\providecommand{\eps}{}
\renewcommand{\eps}{\epsilon}
\providecommand{\ora}[1]{}
\renewcommand{\ora}[1]{\overrightarrow{#1}}

\reversemarginpar
\newtheorem{thm}{Theorem}
\newtheorem{lem}[thm]{Lemma}
\newtheorem{prop}[thm]{Proposition}

\DeclareRobustCommand{\SkipTocEntry}[5]{} 

\newcommand{\sgn}{\mathrm{sgn}}

\numberwithin{equation}{section}
\numberwithin{thm}{section}

\begin{document}

\title{Hipster random walks} 
\author[L.\ Addario-Berry]{Louigi Addario-Berry}
\address{Department of Mathematics and Statistics, McGill University, Montr\'eal, Canada}
\email{louigi.addario@mcgill.ca}

\author[H.\ Cairns]{Hannah Cairns}
\address{Department of Mathematics, Cornell University, Ithaca, USA}
\email{ahc238@cornell.edu}

\author[L.\ Devroye]{Luc Devroye}
\address{School of Computer Science, McGill University, Montr\'eal, Canada}
\email{lucdevroye@gmail.com}

\author[C.\ Kerriou]{Celine Kerriou}
\address{Department of Mathematics and Statistics, McGill University, Montr\'eal, Canada}
\email{celine.kerriou@mail.mcgill.ca}

\author[R.\ Mitchell]{Rivka Mitchell} 
\address{Department of Mathematics and Statistics, McGill University, Montr\'eal, Canada}
\email{rivka.mitchell@mail.mcgill.ca}

\date{August 28, 2019} 


\begin{abstract} 
We introduce and study a family of random processes on trees we call {\em hipster random walks}, special instances of which we heuristically connect to the $\mathrm{min}$-$\mathrm{plus}$ binary trees introduced by Pemantle \cite{pemantlecourant} and studied by Auffinger and Cable, and to the critical random hierarchical lattice studied by \citet{MR2079916}. We prove distributional convergence for the processes by showing that their evolutions can be understood as a discrete analogues of certain convection-diffusion equations, then using a combination of coupling arguments and results from the numerical analysis literature on convergence of numerical approximations of \textsc{pde}s. 
\end{abstract}

\maketitle



\section{\bf Introduction}\label{sec:intro} 
Let $\cT$ denote the complete rooted infinite binary tree. The root receives label $\emptyset$; children of node $v$ receive labels $v0$ and $v1$. In this way generation-$n$ nodes of $\cT$ are labeled by the set $\cL_n := \{0,1\}^n$. For $n \ge 1$, write $\cT_n$ for the binary tree consisting of the root of $\cT$ and its first $n$ generations of descendants. The {\em leaves} of $\cT_n$ are the nodes $\cL_n$ in generation $n$; its {\em internal nodes} are precisely the nodes of $\cT_{n-1}$. 

Fix any assignment $F=(f_v, v \in \cT)$ of binary functions $f_v:\R\times \R\to\R$ to the nodes of $\cT$. Then for any $n\ge 1$, the functions $F_n=(f_v,v \in \cT_n)$ may be viewed as turning $\cT_n$ into a recursively constructed function of arity $2^n$, with inputs at the leaves of $\cT_n$ and output at the root of $\cT_n$. More precisely, given real values $\vec{z} = (z_v,v \in \cT)$ for $n \ge 1$ let $F_n^{\vec{z}}: \cT_n \to \R$ be given by 
\begin{equation}\label{eq:fundamental_recursion}
F_n^{\vec{z}}(v) = \begin{cases} z_v	&\mbox{ if }v \in \cL_n \\
								f_v(F_n^{\vec{z}}(v0),F_n^{\vec{z}}(v1))	& \mbox{ if }u \in \cT_{n-1}\, .
						\end{cases}
\end{equation}
When either the functions comprising $F$ are random, or the inputs are random (or both), then $F_n$ is itself a random function; a large number of problems in probability can be phrased in terms of such random functions. As a very simple example, fix $p \in (0,1)$, and independently for each $v \in \cT$ define $f_v$ by 
\[
f_v(x,y) = \begin{cases} 
		1		& \mbox{ with probability }p, \\
		x+y+1 	& \mbox{ with probability }1-p. 
		\end{cases} 
\] 
Then $F_n^{\vec{1}}(\emptyset)$ is distributed as the total number of individuals in the first $n$ generations of a branching process with offspring distribution given by $p_0=p$, $p_2=1-p$. Here $\vec{1}$ assigns value $1$ to all nodes of $\cT$; below we likewise write $\vec{0}$ for the vector assigning values $0$ to all nodes of $\cT$. 

This work establishes two distributional limit theorems for the output of systems of random functions on $\cT$ which we dub {\em hipster random walks}.

\smallskip
\noindent {\bf Hipster random walk.} 
Let $(A_v,v \in \cT)$ be independent Bernoulli$(1/2)$ random variables, and let $(D_v,v \in \cT)$ be \textsc{iid} random variables, independent of $(A_v,v \in \cT)$. Then set 
\begin{equation}\label{eq:hrw_recursion}
f_v(x,y) = 
xA_v \I{x\ne y} + y(1-A_v)\I{x\ne y} + (x+D_v)\I{x = y} . 
\end{equation}
In other words, if $x\ne y$ then $f_v(x,y)$ flips a fair coin to decide whether to output $x$ or $y$. If $x=y$ then $f_v(x,y)$ outputs $x+D_v$. We call $(D_v,v \in \cT)$ the {\em steps} of the hipster random walk.

\smallskip
The name is inspired by the following intuitive picture, which is based on the stereotype that hipsters don't want to be observed liking popular things. In our setting, the ``things'' in question are potential random walk locations. Imagine that for each node $v$, one of $v0$ or $v1$ is hipper than the other; which one is hipper is determined randomly using $A_v$.   If $x\ne y$ then the hipper individual doesn't have any new company at their current location and stays put. If $x=y$ then the hipper individual detects new company, takes this as a sign that their current location is becoming popular, and so decides to leave (moves to $x+D_v$). The output of $f_v(x,y)$ is the new location of the hipper of $v0$ and $v1$. 

In this work we focus on two specific choices for the common law of the steps.
\smallskip

\noindent {\bf Totally asymmetric $q$-lazy simple hipster random walk.} This is the hipster random walk with steps $(C_v,v \in \cT)$ which are independent Bernoulli$(q)$ random variables. 
In this case, definition (\ref{eq:fundamental_recursion}) yields functions $B_n^{\vec{z}}:\cT_n \to \R$ given by 
\begin{equation}\label{eq:talhrw}
B_n^{\vec{z}}(v) = \begin{cases}
z_v							&\mbox{ if }v \in \cL_n \\
B_{n}^{\vec{z}}(v0) A_v + B_{n}^{\vec{z}}(v1) (1-A_v)	
& \mbox{ if }v \in \cT_{n-1}, B_{n}^{\vec{z}}(v0) \neq B_{n}^{\vec{z}}(v1)\\
 B_{n}^{\vec{z}}(v0) + C_v
 & \mbox{ if }v \in \cT_{n-1}, B_{n}^{\vec{z}}(v0) = B_{n}^{\vec{z}}(v1)\, .
						\end{cases}
\end{equation}

\smallskip
\noindent {\bf Symmetric simple hipster random walk.}  This is the hipster random walk with steps $(R_v,v \in \cT)$ satisfying $\p{R_v=1}=\p{R_v=-1}=1/2$. In this case,  definition (\ref{eq:fundamental_recursion}) yields functions $G_n^{\vec{z}}:\cT_n \to \R$ given by 
\begin{equation}\label{eq:sshrw}
G_n^{\vec{z}}(v) = \begin{cases}
 z_v							&\mbox{ if }v \in \cL_n \\
G_{n}^{\vec{z}}(v0) A_v + G_{n}^{\vec{z}}(v1) (1-A_v)	
& \mbox{ if }v \in \cT_{n-1}, G_{n}^{\vec{z}}(v0) \neq G_{n}^{\vec{z}}(v1)\\
 G_{n}^{\vec{z}}(v0) + R_v
 & \mbox{ if }v \in \cT_{n-1}, G_{n}^{\vec{z}}(v0) = G_{n}^{\vec{z}}(v1)\, .
						\end{cases}
\end{equation}

\noindent Our main results are contained in the following two theorems. 
\begin{thm}\label{thm:main1}
Let $\vec{Z}=(Z_v,v \in \cT)$ be \textsc{iid} integer random variables. 
Next fix $q \in (0,1)$, and for $n \ge 1$ let $B_n = B_n^{\vec{Z}}(\emptyset)$ be the output of the $n$-step totally asymmetric $q$-lazy hipster random walk on input $(Z_v,v \in \cL_n)$. 
Then
\[
\frac{B_n}{(4q\cdot n)^{1/2}} \stackrel{\mathrm{d}}{\longrightarrow} B\, ,
\] 
where $B$ is $\mathrm{Beta}(2,1)$-distributed.
\end{thm}
\begin{thm}\label{thm:main2}
Let $\vec{Z}=(Z_v,v \in \cT)$ be \textsc{iid} integer random variables, and 
for $n \ge 1$ let $G_n=G_n^{\vec{Z}}(\emptyset)$ be the output of the $n$-step symmetric simple hipster random walk on input $(Z_v,v \in \cL_n)$. Then 
\[
(36n)^{-1/3} G_n + 1/2 \stackrel{\mathrm{d}}{\longrightarrow} G\, ,
\] 
where $G$ is $\mathrm{Beta}(2,2)$-distributed.
\end{thm}

\subsection*{Related work}

Fix $p \in (0,1)$ and independently for each $v \in \cT$, define $f_v$ by
\[
f_v(x,y) = \begin{cases} 
		x+y	& \mbox{ with probability }p, \\
		\min(x,y)	& \mbox{ with probability }1-p.
		\end{cases}
\]
Write $M_n^{\vec{z}}$ for the resulting random functions (again obtained by applying definition (\ref{eq:fundamental_recursion}). The study of this model was proposed by Robin Pemantle \cite{pemantlecourant}, who conjectured that when $p=1/2$, 
\begin{equation}\label{eq:minplus}
\frac{\ln M_n^{\vec{1}}(\emptyset)}{(\pi^2 n/3)^{1/2}} \stackrel{\mathrm{d}}{\longrightarrow} B,
\end{equation}
where $B$ is Beta(2,1)-distributed. This conjecture was recently proved by Auffinger and Cable \cite{auffingercable2018}, who dubbed this model {\em Pemantle's min-plus binary tree}.

There is an obvious similarity between (\ref{eq:minplus}) and the convergence in Theorem~\ref{thm:main1} and, indeed, there is a heuristic connection between the models.  By (\ref{eq:minplus}), we know that $M_n$ is growing at a stretched exponential rate. In view of this, it is natural to consider what is happening at a log scale.  Write $\beta_0 = \log M_n^{\vec{1}}(0)$ and $\beta_1 = \log M_n^{\vec{1}}(1)$ for the log (base 2, here and below) of the inputs to the root. 

Most of the time $\log M_n^{\vec{1}}(0)$ and $\log M_n^{\vec{1}}(1)$ take radically different values (since they are identically distributed and exhibit random fluctuations on the scale $n^{1/2}$). In this (typical) case, with $\beta_\emptyset=\log M_n^{\vec{1}}(\emptyset)$ we have 
\[
\beta_\emptyset = 
\begin{cases} 
		\log (M_n^{\vec{1}}(0)+M_n^{\vec{1}}(1)) \approx \max(\beta_0,\beta_1)	& \mbox{ with probability }1/2, \\
		\log \min(M_n^{\vec{1}}(0),M_n^{\vec{1}}(1))\approx \min(\beta_0,\beta_1) 	& \mbox{ with probability }1/2.
		\end{cases}			
\]
In other words, when $M_n^{\vec{1}}(0)$ and $M_n^{\vec{1}}(1)$ are extremely different, the output at the root just looks like the value of a random child. 

On the other hand, will occasionally have $M_n^{\vec{1}}(0) \approx M_n^{\vec{1}}(1)$. In this case, the dynamics look rather different; for example, when  $M_n^{\vec{1}}(0) = M_n^{\vec{1}}(1)$ we have 
\[
\beta_\emptyset = 
\begin{cases} 
		\log (M_n^{\vec{1}}(0)+M_n^{\vec{1}}(1)) = \beta_0+1	& \mbox{ with probability }1/2, \\
		\log \min(M_n^{\vec{1}}(0),M_n^{\vec{1}}(1))= \beta_0	& \mbox{ with probability }1/2.
		\end{cases}
\]
So in this case, at the log scale, the output of the min-plus binary tree just looks like the common value of the children plus a Bernoulli$(1/2)$ increment. This looks very much like the dichotomy for the totally asymmetric hipster random walk: when the children have different values, output the value of a random child; when they have the same value, output that value plus a random increment.

The analogy isn't perfect, because when $M_n^{\vec{1}}(0)$ and $M_n^{\vec{1}}$ take similar but not identical values, the behaviour of the min-plus binary tree interpolates between the two cases. This ``smearing out'' creates a speed-up relative to the totally asymmetric $1/2$-lazy simple hipster random walk (the constants in the rescaling are $(\log e)\cdot\sqrt{\pi^2/3}$ and $\sqrt{2}$, respectively). 

\smallskip
Another related model, called the {\em random hierarchical lattice}, was proposed by Hambly and Jordan \cite{MR2079916}, which in the language of this work may be described as follows. Fix $p \in (0,1)$, and independently for each $v \in \cT$, define $f_v$ by
\[
f_v(x,y) = \begin{cases} 
		x+y	& \mbox{ with probability }p, \\
		\frac{xy}{x+y}	& \mbox{ with probability }1-p. 
		\end{cases}
\]
A natural interpretation of this is as follows. View the inputs to $v$ as electrical networks with effective resistances $x$ and $y$. Then at node $v$ the resistors are combined in series or in parallel, with probability $p$ or $1-p$ respectively; the output is the new, combined network. 

Write $R_n^{\vec{z}}$ for the resulting random functions. Hambly and Jordan show that almost surely, $R_n^{\vec{1}}(\emptyset) \to \infty$ when $p > 1/2$ and $R_n^{\vec{1}}(\emptyset) \to 0$ when $p < 1/2$, and conjecture that $R_n(p)$ almost surely grows exponentially when $p=1/2$. 

By analogy with the min-plus tree, it seems plausible to conjecture that when $p=1/2$, the random variables $R_n^{\vec{1}}(\emptyset)$ are again growing at a stretched exponential scale. In order to make a more precise guess at the phenomenology, we reprise the argument from the case of the min-plus binary tree. 

In the current setting, if $|\log x - \log y|$ is large then $\log(x+y) \approx \max(\log(x),\log(y))$, and $\log(xy/(x+y)) \approx \min(\log(x),\log(y))$. In this case, $\log f_v(x,y)$ just looks like the value of a random child. On the other hand, when $x=y$ then $\log(x+y)=\log(x) + 1$ and $\log(xy/(x+y)) = \log(x)-1$, so $\log f_v(x,y)$ looks like the common $\log$-value of the children plus a random walk step. Again, there is an interpolation between these two extremes, but the heuristic suggests that at a log scale, when $p=1/2$ the random hierarchical lattice should look somewhat like a symmetric simple hipster random walk. In view of the validity of such a heuristic in the case of Pemantle's min-plus binary tree, we are led to conjecture that there exists $c > 0$ such that (in the case $p=1/2$), with $R_n = R_n^{\vec{1}}(\emptyset)$,  
\[
\frac{\log R_n}{(cn)^{1/3}}+1/2 \stackrel{\mathrm{d}}{\longrightarrow} R,
\]
where $R$ is Beta$(2,2)$-distributed. This is in slight disagreement with a prediction of Hambly and Jordan \cite{MR2079916}, who write ``In the case $p=1/2$ ... we also believe that there is an almost sure exponential growth rate for the resistance''. The truth of our conjecture would imply that the growth rate is in fact stretched exponential. However, even the weaker conjecture that 
\[
R_n \convdist \frac{1}{2} \delta_0 + \frac{1}{2} \delta_\infty,
\]
i.e. that $\p{R_n \le r} \to 1/2$ for any $r \in (0,\infty)$, is open at this point. 

\subsection*{Our approach}
Recall that $B_n^{\vec{z}}$ and $G_n^{\vec{z}}$ are our notation for the totally asymmetric and simple hipster random walks, respectively. 
For a probability measure $\mu$ supported by $\Z$ we write $\mathbf{P}_\mu$ for the probability measure under which the the entries of $\vec{z}$ are \textsc{iid} with law $\mu$, independent of all the other random variables describing the system. 

The utility of taking random input values $\vec{z}$ is that the measures $\mathbf{P}_{\mu}$ endow the random functions $B_n^{\vec{z}}$ and $G_n^{\vec{z}}$ with a sort of projective consistency: under $\mathbf{P}_{\mu}$, for all $0 \le j \le n$, the random variables $(B_n^{\vec{z}}(v),v \in \cL_j)$ are \textsc{iid} with the law of $B_{n-j}^{\vec{z}}(\emptyset)$. An exactly analogous statement holds for $G_n^{\vec{z}}$. This allows us to write recurrences in $n$ for the output distribution at the root in both models. 

We first derive the recurrence for the totally asymmetric $q$-lazy simple hipster random walk. Write 
\[ 
p^n_j(\mu) = \psub{\mu}{B_n^{\vec{z}}(\emptyset)=j}\,. 
\]
Provided we assume that $\mu$ is supported on $\Z$, then the distribution of $B_n^{\vec{z}}(\emptyset)$ under $\mathbf{P}_{\mu}$ is determined by the values $(p^n_j({\mu}))_{n \in \Z}$. 

In what follows, we'll write $B_n(v)=B_n^{\vec{z}}(v)$ when the law of $\vec{z}$ is clear from context. By considering the values at the children of the root, using (\ref{eq:talhrw}) we have 
\begin{align*} 
p^{n+1}_{j}(\mu) 
& = \psub{\mu}{B_{n+1}(\emptyset)=j} \\
& = \frac{1}{2}\psub{\mu}{B_{n+1}(0) = j}\psub{\mu}{B_{n+1}(1) \neq j} + \frac{1}{2}\psub{\mu}{B_{n+1}(0) \neq j}\psub{\mu}{B_{n+1}(1) = j}\\
&+  q\cdot\psub{\mu}{B_{n+1}(0) = B_{n+1}(1) = j-1} + (1-q)\cdot\psub{\mu}{B_{n+1}(0) = B_{n+1}(1) = j}\\
&= p_j^n(\mu)(1-p_j^n(\mu)) + q \cdot (p_{j-1}^n(\mu))^2 + (1-q)\cdot (p_j^n(\mu))^2. 
\end{align*}
For the second equality we use the fact that, under $\mathbf{P}_{\mu}$, $B_{n+1}(0)$ and $B_{n+1}(1)$ are independent and have the law of $B_n(\emptyset)$. After rearrangement this yields the identity 
\begin{equation}\label{eq:tal_rec}
p^{n+1}_j(\mu) - p^{n}_j(\mu) = -q\cdot(p^{n}_{j}(\mu)^2 - p^{n}_{j-1}(\mu)^2).\,
\end{equation}
This is a discrete analog of the inviscid Burgers' equation, 
\begin{equation}\label{eq:burgers}
\partial_t u = - q \cdot \partial_x(u^2).
\end{equation}

Next consider the symmetric simple hipster random walk, and write 
\nomenclature[qnjmu]{$q^n_j(\mu)$}{$\psub{\mu}{G_n^{\vec{z}}(\emptyset)=j}$; probability that $n$-level SSHRW outputs $j$ on $\mu$-distributed input.}
\[
q^n_j(\mu) = \psub{\mu}{G_n^{\vec{z}}(\emptyset)=j}\,.
\]
We again write $G_n(v)=G_n^{\vec{z}}(v)$ when the law of $\vec{z}$ is clear. 
By considering the values at the children of the root, 
using (\ref{eq:sshrw}) have 
\begin{align*} 
q^{n+1}_{j}(\mu) &= \psub{\mu}{G_{n+1}(\emptyset)=j} \\
& = \frac{1}{2}\psub{\mu}{G_{n+1}(0) = j}\psub{\mu}{G_{n+1}(1) \neq j} + \frac{1}{2}\psub{\mu}{G_{n+1}(0) \neq j}\psub{\mu}{G_{n+1}(1) = j}\\
&+  \frac{1}{2}\cdot\psub{\mu}{G_{n+1}(0) = G_{n+1}(1) = j-1} + \frac{1}{2}\cdot\psub{\mu}{G_{n+1}(0) = G_{n+1}(1) = j+1}\\
&= q_j^n(\mu)(1-q_j^n(\mu)) + \frac{1}{2} q_{j-1}^n(\mu)^2 + \frac{1}{2} q_{j+1}^n(\mu)^2,
\end{align*}
where for the second equality we have again used projective consistency. Rearrangement now gives 
\begin{align}\label{eq:ss_rec}
q^{n+1}_j(\mu) - q^{n}_j(\mu) & = \frac{1}{2} \left((q^n_{j+1}(\mu))^2 - 2(q_{j}^n(\mu))^2 + (q_{j+1}^n(\mu))^2\right), 
\end{align}
a discrete analogue of the porous medium equation for groundwater infiltration
\cite{Vazquez2006},
\begin{equation}\label{eq:boussinesq}
\partial_t u = \frac{1}{2}\cdot \partial_{xx}(u^2) \, .
\end{equation}

The preceding development shows that for both of the models under consideration, the evolution of the probability distribution as $n$ varies is a discrete analogue of a \textsc{pde}. As such, it's natural to expect the behaviour of the \textsc{pde} to predict that of the finite system. Indeed, if $B$ is Beta$(2,1)$-distributed then $(4qt)^{1/2}B$ has density 
\begin{equation}\label{eq:burgers_sol}
\frac{x}{2qt} \I{0 \le x \le (4qt)^{1/2}}\, ,
\end{equation}
which solves (\ref{eq:burgers}) at its points of differentiability. Similarly, if $G$ is Beta$(2,2)$-distributed then $(36t)^{1/3}(G-1/2)$ is supported by $[-(36t)^{1/3},(36t)^{1/3}]$ and has density
\begin{equation}\label{eq:pme_sol}
\frac{1}{(36t)^{1/3}} \cdot 6 \left(\frac{1}{2}+\frac{x}{(36t)^{1/3}}\right)\left(\frac{1}{2}-\frac{x}{(36t)^{1/3}}\right) = 
\frac{3}{4}\left(\Big(\frac{2}{9t}\Big)^\frac{1}{3} - \Big(\frac{2x^2}{9t}\Big)\right)\,
\end{equation}
on that interval; this solves (\ref{eq:boussinesq}) wherever it is differentiable. 

It isn't {\em a priori} obvious that this perspective is useful, for multiple reasons. First, the \textsc{pde}s under consideration are degenerate convection-diffusion equations, for which neither existence nor uniqueness of solutions is clear. (The ``solutions'' above already have points of non-differentiability so do not make sense classically; on the other hand, once one abandons classical solutions uniqueness is in general lost.) Second, even if one can identify the ``correct'' \textsc{pde} solutions, it isn't obvious that the behaviour of the finite systems will correctly approximate the limiting \textsc{pde}s. 

Showing that a discrete difference equation provides a good approximation for an associated \textsc{pde} is a problem that sits squarely within the area of numerical analysis. It turns out that, by viewing (\ref{eq:tal_rec}) and (\ref{eq:ss_rec}) as  numerical approximation schemes we are able to use results from the rigorous numerical analysis literature to prove distributional convergence of the associated random variables. 

In fact, in their initial form the numerical analysis results are not strong enough for our purposes, as they establish convergence in an integrated sense which doesn't give us access to the distribution of the hipster random walks at fixed times. However, we are able to strengthen the numerical approximation theorems using coupling arguments together with carefully chosen initial conditions for the hipster random walks. The coupling arguments are slightly surprising, so we briefly describe them. Write $\mu$ for the law of the entries of the input field $\vec{Z}$. Suppose $\vec{Z}$ is replaced by another input field $\vec{W}$ whose entries have some law $\nu$, and that for some $\alpha \in (0,1)$ there exists a coupling $(z,w)$ of $\mu$ and $\nu$ such that $\p{z > w} \le \alpha$. Then the totally asymmetric dynamics (\ref{eq:talhrw}) may be coupled so that for any $n$, 
\[
\p{B_n^{\vec{Z}} > B_n^{\vec{W}}}\le \alpha, 
\]
and likewise the symmetric dynamics (\ref{eq:sshrw}) may be coupled so that for any $n$, 
\[
\p{G_n^{\vec{Z}} > G_n^{\vec{W}}}\le \alpha.  
\]
For the precise statements, see Lemmas~\ref{stochdom_talshrw} and ~\ref{stochdom}, below. 

The remainder of the paper is structured as follows. In Section~\ref{sec:na} we describe the setting of the numerical approximation theorem we will use, as well as the theorem itself (Theorem~\ref{thm:evje_karlsen}). We also state propositions which verify that the dynamics we study may be recast within the framework of Theorem~\ref{thm:evje_karlsen}; the proofs of these propositions appear in Appendix~\ref{bv_proofs}. In Section~\ref{sec:integrated} we prove ``integrated'' versions of Theorems~\ref{thm:main1} and~\ref{thm:main2}. In Section~\ref{sec:proofs} we state the coupling lemmas mentioned above (their proofs are also deferred to Appendix~\ref{bv_proofs}), and use them to prove Theorems~\ref{thm:main1} and~\ref{thm:main2}. Finally, Section~\ref{sec:conc} contains several suggestions for interesting avenues of research related to hipster random walks and their ilk. 

\subsection*{\bf Acknowledgements}
We thank Lia Bronsard for directing us to the reference \cite{ek}. We further thank Rustum Choksi, Jessica Lin, Pascal Maillard and Robin Vacus for useful conversations.

\section{\bf Finite approximation schemes for degenerate convection-diffusion equations} \label{sec:na}

This section summarizes the principal result of \cite{ek}, which is the main tool we use to study the asymptotic behaviour of the recurrence relations (\ref{eq:burgers}) and (\ref{eq:boussinesq}). In \cite{ek}, Evje and Karlsen consider convection-diffusion initial value problems of the form
\begin{equation}\label{eq:convection_diffusion}
\begin{cases}
\partial_t u + \partial_x f(u) - \partial_{xx} K(u) = 0, & (x,t) \in \R \times (0,T) \\
u(x,0)= u_0(x). 
\end{cases}
\end{equation}
The problem is specified by the choice of the (measurable) functions $f:\R \to \R$ and $K:\R \to [0,\infty)$, which are respectively called the {\em convection flux} and the {\em diffusion flux}, and by the choice of initial condition $u_0: \R \rightarrow \R$. 
Burgers' equation (\ref{eq:burgers}) is obtained by taking $f(u)=f_{\mathrm{B}}^{(q)}(u):=q u^2$ and $K = K_{\mathrm{B}}^{(q)} \equiv 0$. The porous membrane equation (\ref{eq:boussinesq}) is obtained by taking $f = f_{\mathrm{P}} \equiv 0$ and $K(u)=K_{\mathrm{P}}(u) := \frac{1}{2}u^2$.

Evje and Karlsen provide sufficient conditions for the convergence of certain numerical approximation schemes to solutions of (\ref{eq:convection_diffusion}). The ``solutions'' in question are not everywhere differentiable, so must be understood in a weak sense, which we now explain in detail. (We impose stronger conditions on our solutions than those in \cite{ek}, since they are easier to state and hold in the cases we consider in the current work.) 

Recall that for a signed measure $\mu$ on a measurable space $(M,\mathcal{B})$, there are unique non-negative measures $\mu^+,\mu^-$ on $(M,\mathcal{B})$ such that $\mu=\mu^+-\mu^-$; this is the {\em Jordan decomposition} of $\mu$. The {\em variation} of $\mu$ is the (unsigned) measure $|\mu|:=\mu^++\mu^-$. 

A measurable function $z:\R \times [0,T]$ is {\em locally integrable} if for all compact sets $S \subset \R \times [0,T]$, the function $z|_S$ is integrable. It has {\em bounded variation} if it is locally integrable and its partial derivatives $\partial_x z$ and $\partial_t z$, considered as signed Borel measures, satisfy $|\partial_x z|(\R\times (0,T)) + |\partial_t z|(\R \times (0,T)) < \infty$. Finally, $z$ lies in the H\"older space $C^{1,\frac 1 2}(\R \times [0,T])$ if it is bounded and additionally there is $C>0$ such that for all $(x,s),(y,t) \in \R \times [0,T]$, 
\[ 
|z(x,s) - z(y,t)| \le C (|y-x| + |t-s|^{1/2})\, . 
\] 

Let $u: \R \times [0,T] \to \R$ be a bounded measurable function. We say $u$ is a solution of (\ref{eq:convection_diffusion}) if the following conditions hold. 
\begin{enumerate} 
\item The function $u|_{\R \times (0,T)}$ has bounded variation and $u(\cdot,0)\equiv u_0$. 
\item The function $K(u)$ is bounded and Lipschitz continuous on $\R \times [0,T]$. 
\item For all non-negative $\phi \in C^\infty(\R \times [0,T])$ with compact support and with $\phi|_{t=T} \equiv 0$, and for all $c \in \R$,
\begin{align}\label{eq:entropy_ineq}
\int_{\R}\int_{(0,T)}  
\left(
(u-c) \cdot \partial_t \phi +(f(u) -f(c) - \partial_x K(u)) \cdot \partial_x\phi
\right)\cdot \sgn(u-c)
\mathrm{d}t
\mathrm{d}x & \nonumber\\
+ 
\int_\R |u_0-c| \phi(x,0)\mathrm{d}x & 
\ge 0
\end{align}
\end{enumerate}
Here and elsewhere, $\sgn(x) := \I{x > 0}-\I{x < 0}$. In \cite{ek}, such a function $u$ is called a {\em BV entropy weak solution} of (\ref{eq:convection_diffusion}). 

The following may help understand the content of (\ref{eq:entropy_ineq}). Imagine that a smooth solution $u$ of (\ref{eq:convection_diffusion}) existed, and fix any bounded smooth function $\phi:\R\times(0,T) \to \R$ with compact support. 
Using integration by parts, we have 
\begin{align*}
\int_{\R}\int_{(0,T)}u\cdot \partial_t \phi \mathrm{d}t\mathrm{d}x  
& = \int_\R \left[u(x,\cdot)\phi(x,\cdot)\right]_{t=0}^T\mathrm{d}x - \int_\R\int_{(0,T)} \phi \cdot \partial_t u \mathrm{d}x \mathrm{d}t. 
\end{align*}
By (\ref{eq:convection_diffusion}) we also have $\partial_t u = \partial_{xx}K(u)-\partial_x f(u)$.
Thus, if $\phi|_{t=T}=0$ then the right-hand side is
\begin{align*}
& -\int_\R u_0\cdot \phi(x,0) \mathrm{d}x - \int_{(0,T)} \int_\R \phi \cdot \left(\partial_{xx}K(u)-\partial_x f(u)\right) \mathrm{d}x \\
& 
=-\int_\R u_0\cdot \phi(x,0) \mathrm{d}x - \int_{(0,T)}\int_\R \partial_x \phi\cdot (\partial_x K(u) - f(u)) \mathrm{d}t \mathrm{d}x,
\end{align*}
the equality following from integration by parts and the fact that $\phi$ has compact support. This yields the identity
\begin{equation}\label{eq:mock_identity}
\int_\R \int_{(0,T)} u \cdot \partial_t \phi + (f(u) - \partial_x K(u)) \cdot \partial_x \phi~\mathrm{d}t \mathrm{d}x +\int_\R u_0 \cdot \phi(x,0) \mathrm{d}x= 0, 
\end{equation}
which is an integrated form of (\ref{eq:convection_diffusion}). Unfortunately, for many \textsc{pde}s, there is no solution of (\ref{eq:convection_diffusion}) in the classical sense as the ``obvious'' candidate is non-differentiable. On the other hand, passing to the integrated form yields too much flexibility --- solutions exist but are not unique. 

One common way to single out a ``physically relevant'' solution of (\ref{eq:convection_diffusion}) is to first add a diffusive term $\eps u_{xx}$ to the \textsc{pde} in (\ref{eq:convection_diffusion}), and find integrated solutions $u^{(\eps)}$ to the modified \textsc{pde}. The smoothing effect of the diffusive term will often yield uniqueness of $u^{(\eps)}$; one may then hope to define $u := \lim_{\eps \downarrow 0} u^{(\eps)}$. Informally, the addition of such a viscosity term is meant to enforce that any ``shocks'' (discontinuities of the solution or of its derivatives) propagate at ``physically meaningful speeds''.

There are many versions of such arguments for different families of \textsc{pde}s; one of the casualties of this approach is that the that the equality in (\ref{eq:mock_identity}) does not always persist in the $\eps \to 0$ limit. Its replacement by an inequality in some sense encodes the idea that shocks inhibit information transmission (i.e.\ they are entropy increasing), but we have not found a  convincing informal explanation of why this is so. For further details on and applications of this perspective, we refer the reader to \cite{MR1410987,MR0413526,MR0350216,MR0267257,MR0216338,MR1785683}. 

For a given initial condition $u_0: \R \rightarrow \R$ and real $\Delta_x,\Delta_t > 0$, we define $(U^n_j)_{n \in \N,j \in \Z}=(U^n_j(u_0,\Delta_x,\Delta_t))_{n \in \N,j \in \Z}$ via the following discretization of (\ref{eq:convection_diffusion}). 
\begin{equation}
\begin{aligned}
 \frac{U^{n+1}_j - U^n_j}{\Delta_t}  & = -\frac{f(U^n_j)-f(U^n_{j-1})}{\Delta_x} + \frac{K(U^n_{j+1}) - 2 K(U^n_j)+K(U^n_{j-1})}{(\Delta_x)^2}\, , & j \in \Z, n \ge 1 \\
 U^0_j & = \frac{1}{\Delta_x} \int_{j\Delta_x}^{(j+1)\Delta_x} u_0(x)\mathrm{d}x, & j \in \Z\, .
 \end{aligned}
\label{eq:schemedef}
\end{equation}
We refer to this as an {\em approximation scheme} for (\ref{eq:convection_diffusion}). 
Given an interval $I \subseteq \R$, we say the approximation scheme is {\em monotone on $I$} if the function $S:\R^3 \to \R$ defined by 

\begin{equation} \label{monotone_scheme}
    S(u^-,u,u^+) = \frac{u}{\Delta_t} - \frac{1}{\Delta_x}(f(u)-f(u^-)) + \frac{1}{(\Delta_x)^2}(K(u^+)-2K(u)+K(u^-))\, 
\end{equation}
satisfies $S(I\times I \times I) \subseteq I \times I \times I$, and is non-decreasing in each argument on $I\times I \times I$.
Equivalently, in the first equation in (\ref{eq:schemedef}), the value of $U^{n+1}_j$ is a monotone function of $U^n_{j+1},U^n_j$ and $U^n_{j-1}$, provided those values all lie in $I$. 
\begin{thm}[\cite{ek}]\label{thm:evje_karlsen}
Suppose $f: \R \rightarrow \R$ and $K: \R \rightarrow \R$ are continuously differentiable. Fix a bounded variation function $u_0: \R \to \R$ with compact support and such that $f(u_0)-K'(u_0)$ also has bounded variation. Then there is a unique BV entropy weak solution $u$ of the corresponding convection-diffusion equation (\ref{eq:convection_diffusion}). 

Next, fix sequences $(\Delta_x^M)_{M \ge 1}$ and $(\Delta_t^M)_{M \ge 1}$ decreasing to zero, such that the corresponding approximation schemes $(U^n_j(u_0,\Delta_x^M,\Delta_t^M))_{n \in \N,j \in \Z}$ are monotone on an interval $I \subseteq \R$. 
Let $u^M:\R\times [0,\infty) \to \R$ be the function which takes the value $U^n_j(u_0,\Delta_x^M,\Delta_t^M)$ on the half-open rectangle 
\[
[j\Delta^M_x,(j+1)\Delta^M_x)\times[n\Delta^M_t,(n+1)\Delta^M_t). 
\]

If $u_0(\R) \subseteq I$, then $u^m$ converges pointwise almost everywhere to $u$, and for all compacts $C \subset \R \times [0,\infty)$, $\int_C |u^M-u| \mathrm{d}x\mathrm{d}t \to 0$ as $M \to \infty$. Moreover, the sequence of functions $(K(u^M))_{M \ge 1}$ converges uniformly on compacts to $K(u)$.
\end{thm}
In \cite{ek}, the approximation schemes are required to be monotone on $\R$; however, the above formulation is in fact an immediate consequence of the proof in \cite{ek}. 

The next two propositions verify that (\ref{eq:burgers_sol}) and (\ref{eq:pme_sol}) indeed describe the the BV entropy solutions of the convection-diffusion equation~\ref{eq:convection_diffusion}, for the relevant choices of $f$ and $K$, and that the corresponding approximation schemes are monotone provided we take a suitably fine-meshed discretization. The proofs of these propositions, which boil down to careful applications of the divergence theorem together with case analysis (based on the value of $c$ in (\ref{eq:entropy_ineq})), appear in Appendix~\ref{bv_proofs}.

The first of the propositions relates to Burgers' equation, which corresponds to the totally asymmetric hipster random walk.
For this model $f(u) = f_{\mathrm{B}}(u) := q\cdot u^2$ and $K = K_{\mathrm{B}} \equiv 0$.
\begin{prop}\label{prop:trafficscheme}
Let $q\in (0,1)$. Fix $\varepsilon>0$ and $T>0$ and define  $u_{\mathrm{B}} : \R \times [0,T] \rightarrow \R$ by
\[ u_{\mathrm{B}}(x,t) = \frac{x}{2q(t + \varepsilon)}\I{x\in[0, \sqrt{4q(t+\varepsilon)}]}\]
Then $u_\mathrm{B}$ is the BV entropy weak solution to the initial value problem 
\[
\partial_t u + \partial_x (q\cdot u^2) = 0,
\]
with initial condition $u_0(x) = u_{\mathrm{B}}(x,0)$. 
Moreover, the following holds. Fix $M >0$, let $\Delta_x^M = \frac{1}{M}$ and $\Delta_t^M = \frac{1}{M^2}$, and consider the approximation scheme $(U^n_j(u_0,\Delta_x^M,\Delta_t^M))_{n \in \N,j \in \Z}$ given by 
\begin{equation}\label{eq:b_scheme}
\begin{aligned}
 \frac{U^{n+1}_j - U^n_j}{\Delta_t^M}  & = -q \cdot \left(  \frac{(U_j^n)^2 - (U^n_{j-1})^2}{\Delta_x^M}\right)\, , & j \in \Z, n \ge 1 \\
 U^0_j & = \frac{1}{\Delta^M_x} \int_{j\Delta^M_x}^{(j+1)\Delta^M_x} u_0(x)\mathrm{d}x, & j \in \Z\,
 \end{aligned}
\end{equation}
which is obtained from (\ref{eq:schemedef}) by taking $K\equiv K_{\mathrm{B}}$ and $f\equiv f_{\mathrm{B}}$. 
Then for $M$ sufficiently large, the approximation scheme $(U^n_j(u_0,\Delta_x^M,\Delta_t^M))_{n \in \N,j \in \Z}$ is monotone on $[0,(q\varepsilon)^{-1/2}]$. 
\end{prop}
The second of the propositions concerns the porous medium equation, which corresponds to the symmetric simple hipster random walk. For this model we defined $f = f_{\mathrm{P}}\equiv 0 $ and $K(u) = K_{\mathrm{P}}(u) = \frac{1}{2}u^2.$
\begin{prop}\label{prop:groundwaterscheme}
Fix $\varepsilon>0$ and $T>0$ and define  $v_{\mathrm{P}} : \R \times [0,T] \rightarrow \R$ by
\[v_{\mathrm{P}}(x,t) = \max\left(\frac{3}{4}\left( \left(\frac{2}{9(t+\varepsilon)}\right)^\frac{1}{3} - \left(\frac{2x^2}{9(t+\varepsilon)}\right)\right),0\right).\]
Then $v_{\mathrm{P}}$ is the BV entropy weak solution to the initial value problem 
\[
\partial_t v - \partial_{xx} \frac{v^2}{2} = 0,
\]
with initial condition $v_0(x) = v_{\mathrm{P}}(x,0)$. Moreover, the following holds. Fix $M>0$, let $\Delta_x^M = \frac{1}{M}$ and $\Delta_t^M = \frac{1}{M^3}$, and consider the approximation scheme $(U^n_j(v_0,\Delta_x^M,\Delta_t^M))_{n \in \N,j \in \Z}$ given by
\begin{equation}\label{eq:p_scheme}
\begin{aligned}
\frac{U_j^{n+1} - U_j^n}{\Delta_t^M} &= \frac{(U^n_{j+1})^2 - 2(U_j^n)^2 + (U_{j-1}^n)^2}{2(\Delta_x^M)^2}, & j\in\Z, n\ge 1\\
U_j^0 &= \frac{1}{\Delta_x^M}\int_{j \Delta_x^M}^{(j+1)\Delta_x^M}v_0(x)dx, & j \in \Z 
 \end{aligned}
\end{equation}
which is obtained from (\ref{eq:schemedef}) by taking $K \equiv K_{\mathrm{P}}$, and $f \equiv f_{\mathrm{P}}$. Then for $M$ sufficiently large, the approximation scheme $(U^n_j(v_0, \Delta_x^M, \Delta_t^M))_{n\in\N, j\in\Z}$ is monotone on $[0,(3/4)(2/(9\varepsilon))^{1/3}]$.
\end{prop}

\section{\bf Integrated versions of Theorems~\ref{thm:main1} and~\ref{thm:main2}}\label{sec:integrated}
The approximation schemes in Propositions~\ref{prop:trafficscheme} and~\ref{prop:groundwaterscheme} differ from the recurrences for the hipster random walks, namely~(\ref{eq:tal_rec}) and~(\ref{eq:ss_rec}), by factors involving the spatial and discretizations, $\Delta^M_x$ and $\Delta^M_t$. However, the form of those factors is such that we still have easily verified exact relations between the values values $U^n_j$ and the distributions of the hipster random walks. These relations are summarized in the next two propositions. 
Fix a non-negative measurable function $\rho: \R\to [0,\infty)$ with $\int_\R \rho(x)\mathrm{d}x<\infty$. For $M >0$ and $j \in \Z$, define a measure $\rho^M$ on $\Z$ by 
\begin{equation}\label{eq:timezero}
\rho^M(\{j\})= M\int_{j/M}^{(j+1)/M} \rho(x)\mathrm{d}x. 
\end{equation}
Next, for $j \in \Z$ let $u^0_j=u^0_j(\rho,M) = \rho^M(\{j\})$, and 
for $n \ge 1$ define $(u^n_j)_{j \in \Z}=(u^n_j(\rho,M))_{j \in \Z}$ via the recurrence 
\nomenclature[unj]{$u^n_j$}{Defined via recurrence $M(u^{n+1}_j - u^n_j) = -q \cdot\left( (u^n_j)^2 - (u^n_{j-1})^2\right)$} 
\[
M\cdot (u^{n+1}_j - u^n_j) = -q\cdot\left( (u^n_j)^2 - (u^n_{j-1})^2\right). 
\]
Note that this is equivalent to the recurrence in (\ref{eq:b_scheme}) since, in that recurrence, $\Delta^M_x=1/M$ and $\Delta^M_t=1/M^2$. The following proposition connects the evolution of $u^n_j$ with the totally asymmetric hipster random walk. Its proof is a straightforward inductive argument and is omitted. 
\begin{prop}\label{prop:traffic_induction}
Suppose that $\rho$ is a probability density function on $\R$. Fix $M > 0$ and define a measure $\mu=\mu_{\rho,M}$ on $\Z$ by $\mu(\{j\}) = \int_{j/M}^{(j+1)/M} \rho(x)\mathrm{d}x$.  Then for all $n \in \N$ and $j \in \Z$, 
\[
u^n_j(\rho,M) = M\cdot \psub{\mu}{B_n^{\vec{z}}(\emptyset)=j}.
\]
\end{prop}
Next, fix $M$ and $\rho$ as above, and for $j \in \Z$ let $v^0_j=v^0_j(\rho,M) = \rho^M(\{j\})$, where $\rho^M(\{j\})$ is again given by (\ref{eq:timezero}). Then, for $n \ge 1$, define $(v^n_j)_{j \in \Z}=(v^n_j(\rho,M))_{j \in \Z}$ by the recurrence 
\[
M\cdot (v^{n+1}_j - v^n_j) = \frac{1}{2}\left( (v^n_{j+1})^2-2(v^n_j)^2 + (v^n_{j-1})^2\right). 
\]
This is equivalent to the recurrence in (\ref{eq:p_scheme}), as in (\ref{eq:p_scheme}) we have $\Delta^M_x=1/M$ and $\Delta^M_t=1/M^3$. 
\begin{prop}\label{prop:ground}
Suppose that $\rho$ is a probability density function. Fix any $M > 0$ and define a measure $\mu=\mu_{\rho,M}$ on $\Z$ by $\mu(\{j\}) = \int_{j/M}^{(j+1)/M} \rho(x)\mathrm{d}x$. Then for all $n \in \N$ and $j \in \Z$, 
\[
v^n_j(\rho,M) = M\cdot \psub{\mu}{G_n^{\vec{z}}(\emptyset)=j}.
\]
\end{prop}
The proof of Proposition \ref{prop:ground} is also an easy induction and is omitted. 

Having stated these results, we are prepared to prove weakenings of Theorems~\ref{thm:main1} and~\ref{thm:main2}. We must weaken the theorems in two ways. First, rather than starting from arbitrary inputs, we choose initial distributions which are fine-mesh discretization of the initial conditions for which we understand the solutions to the associated initial value problems. In other words, in the totally asymmetric case we will start from a discretization of a scaled Beta$(2,1)$ distribution, and in the symmetric case we will start from a discretized Beta$(2,2)$ distribution. Second, our conclusions concern the distribution of trees of a {\em random} rather than fixed height. The reason for this is that the almost sure convergence provided by Theorem~\ref{thm:evje_karlsen} is two-dimensional (it concerns the space-time field of values $(U^n_j)_{n \ge 0,j \in \Z}$). Fixing the height of the tree corresponds to considering the \textsc{pde} approximation at a fixed time; but Theorem~\ref{thm:evje_karlsen} doesn't {\em a priori} guarantee the absence of ``pathological'' times at which the discrete approximations are badly-behaved.
\begin{prop}\label{prop:traffic_timeaverage}
Fix $q\in (0,1)$ and $\varepsilon \in (0,1)$. Then for $M>0$ let $\mu^M=\mu_{\varepsilon}^M$ be the probability measure on $\Z$ defined by 
\[
\mu^M(\{j\}) = \int_{j/M}^{(j+1)/M} \frac{x}{2q\varepsilon}\I{x \in [0,\sqrt{4q\varepsilon}]}\mathrm{d}x\, . 
\]
Next fix $0\le \ell < r$ and, under $\mathbf{P}_{\mu^M}$, let $W$ be a Uniform$[\ell,r]$ random variable, independent of $\vec{z}$.
Then 
\[ \frac{B^{\vec{z}}_{\lfloor WM^2\rfloor}}{(4q(W+\varepsilon))^{1/2}M} \convdist B \text{ as } M \rightarrow \infty\] 
where $B$ is a $Beta(2,1)$ random variable.
\end{prop}

The joint law of $W$ and $B^{\vec{z}}_{\lfloor WM^2 \rfloor}$ can be given explicitly as 
\[
\psub{\mu^M}{B^{\vec{z}}_{\lfloor WM^2\rfloor} =j,W \in \mathrm{d}t} =  \frac{\mathrm{d}t}{r - \ell} \psub{\mu^M}{B_{\lfloor tM^2\rfloor}^{\vec{z}}(\emptyset)=j} \I{t \in [\ell,r]}. 
\] 
\begin{proof}[Proof of Proposition~\ref{prop:traffic_timeaverage}]
In the proof we write $\mathbf{P}$ instead of $\mathbf{P}_{\mu^M}$ and 
$B_{k}$ instead of $B_{k}^{\vec{z}}(\emptyset)$, for succinctness. 
For $b \in (0,1)$, we have 
\begin{align}
\p{B_{\lfloor WM^2\rfloor} \le (4q(W+\varepsilon))^{1/2} Mb} 
& = 
\int_{\ell}^{r} \p{B_{\lfloor WM^2\rfloor} \le (4q(W+\varepsilon))^{1/2} Mb,W \in \mathrm{d}t} \nonumber\\
& = \frac{1}{r-\ell} \int_{\ell}^{r} \p{B_{\lfloor tM^2\rfloor}\le (4q(t+\varepsilon))^{1/2}Mb} \mathrm{d}t\, .\label{eq:em_equiv}
\end{align}

Define $u_0(x) = (x/(2q\varepsilon))\I{x \in [0,\sqrt{4q\varepsilon}]}$ and for $0 \le t \le 4q$ let 
\[
u(x,t) = \frac{x}{2q(t+\varepsilon)}\I{x \in [0,\sqrt{4q(t+\varepsilon)}]}. 
\]
Then $u\equiv u_B$ and $u_0(x)=u_B(x,0)$, where $u_B$ is as in Proposition~\ref{prop:trafficscheme}, applied with $T=4q$ and $t_0=\eps$. 
Also write $U^n_j=U^n_j(u_0,\Delta_x^M,\Delta_t^M)$, where 
$(U^n_j(u_0,\Delta_x^M,\Delta_t^M))_{n \in \N,j \in \Z}$ is again as in Proposition~\ref{prop:trafficscheme}. 

Note that $\mu^M(\{j\}) = \int_{j/M}^{(j+1)/M} u_0(x)\mathrm{d}x$, 
and by Proposition~\ref{prop:traffic_induction}, for all $n \in \N$ and $j \ge 0$ we have 
$U^n_j = M\cdot \p{B_{n}=j}$. 
Now let 
 $u^M:\R\times[0,\infty)\rightarrow \R$ be the function which takes the value $U_j^n$ on 
 $[j/M, (j+1)/M)\times[n/M^2, (n+1)/M^2)$ 
 for $n$, $j\in \N$. Then for $b \in [0,1]$ we have 
\begin{align*}
    \int\limits_{\ell}^{r}\int\limits_0^{b(4q(t+\varepsilon))^{1/2}} u^M(x,t)dxdt 
    &= \int\limits_{\ell}^{r}\int\limits_0^{b(4q(t+\varepsilon))^{1/2}} M \cdot \p{B_{\lfloor tM^2\rfloor} = \lfloor Mx \rfloor}dxdt\, ,
\end{align*}
so 
\begin{equation}\label{eq:lowerbd}
    \int\limits_{\ell}^{r}\int\limits_0^{b(4q(t+\varepsilon))^{1/2}} u^M(x,t)dxdt 
    \ge  \int\limits_{\ell}^{r}\p{B_{\lfloor tM^2\rfloor} < \lfloor M b(2(t+\varepsilon))^{1/2}\rfloor} dt
\end{equation}
and
\begin{equation}\label{eq:upperbd}
    \int\limits_{\ell}^{r}\int\limits_0^{b(4q(t+\varepsilon))^{1/2}} u^M(x,t)dxdt 
\le \int\limits_{\ell}^{r}\p{B_{\lfloor tM^2\rfloor} \le \lfloor M b(2(t+\varepsilon))^{1/2}\rfloor}dt\, .
\end{equation}

Since $u_B$ is the solution to the initial value problem $\partial_t u + q\partial_x u^2=0$ with initial condition $u_0 = u_{B,0}$, it follows by Theorem \ref{thm:evje_karlsen} that 
\[ 
\lim_{M\rightarrow\infty} \int\limits_{\ell}^{r}\int\limits_0^{b(4q(t+\varepsilon))^{1/2}} u^M(x,t)\ dxdt = \int\limits_{\ell}^{r}\int\limits_0^{b(4q(t+\varepsilon))^{1/2}} u(x,t) dxdt = \int\limits_{\ell}^{r} b^2 = (r - \ell)b^2\, .
\]
Combining this with (\ref{eq:em_equiv}), (\ref{eq:lowerbd}) and (\ref{eq:upperbd}), we obtain that for all $b \in (0,1)$, 
\[
\lim_{M \to \infty} \p{B_{\lfloor WM^2\rfloor} \le (4q(W+\varepsilon))^{1/2} Mb} = b^2. 
\]
For $B$ a Beta$(2,1)$ random variable, $\p{B \le b} = b^2$, so the result follows. 
\end{proof}
\begin{prop}\label{prop:groundwater_timeaverage}
 Fix $\varepsilon\in (0,1)$. For $M>0$ let $\nu^M = \nu^M_\varepsilon$ be the probability measure on $\Z$ defined by 
\[\nu^M(\{j\}) = \int_{j/M}^{(j+1)/M} \max\left(\frac{3}{4}\left(\left(\frac{2}{9\varepsilon}\right)^{1/3} - \left(\frac{2x^2}{9\varepsilon}\right)\right), 0\right) dx.\]
Next fix $0\le \ell < r$ and, under $\bP_{\nu^M}$, let $W$ be a Uniform$[\ell,r]$ random variable, independent of $\vec{z}$. Then 
\[\frac{G^{\vec{z}}_{\lfloor WM^3\rfloor}}{(36(W+\varepsilon))^{1/3}M} + \frac{1}{2}\convdist G \text{ as } M \rightarrow \infty\]
where $G$ is a Beta$(2,2)$ random variable.\\
\end{prop}
Similarly to the previous case, the joint law of $W$ and $G^{\vec{z}}_{\lfloor WM^3\rfloor}$ is given by
\[\bP_{\nu^M}\{G^{\vec{z}}_{\lfloor WM^3\rfloor} = j, W\in dt\} = \frac{dt}{r-\ell}\bP_{\nu^M}\left\{G_{\lfloor t M^3\rfloor}^{\vec{z}}(\emptyset) = j\right\} \I{t\in [\ell,r]}.\]
\begin{proof}[Proof of~Proposition~\ref{prop:groundwater_timeaverage}]
Again, we write $\bP$ instead of $\bP_{\nu^M}$ and $G_k$ instead of $G_k^{\vec{z}}(\emptyset)$. For $0 \le a < b \le 1$, we have 
\begin{align}
&\bP\left\{\frac{G_{\lfloor WM^3\rfloor}}{(36(W+\varepsilon))^{1/3} M} + \frac{1}{2} \in [a,b]\right\} \nonumber\\
& = \frac{1}{r-\ell} \int_{\ell}^{r} \bP\left\{G_{\lfloor tM^3\rfloor}\in [\lfloor M(a-1/2)c\rfloor,\lfloor M (b-1/2)c\rfloor] \right\}dt,\label{ground1}
\end{align}
where $c = (36(t+\varepsilon))^{1/3}$.

Define $v_0(x) = \max\left(\frac{3}{4}\left(\left(\frac{2}{9\varepsilon}\right)^{1/3} - \left(\frac{2x^2}{9\varepsilon}\right)\right), 0\right)$ and for $t\in [0,\lceil r \rceil] $ let 
\[v(x,t) = \max\left(\frac{3}{4}\left(\left(\frac{2}{9(t+\varepsilon)}\right)^{1/3} - \left(\frac{2x^2}{9(t+\varepsilon)}\right)\right), 0\right).\]

Then $v\equiv v_\mathrm{P}$ and $v_0(x) = v_\mathrm{P}(x,0)$, where $v_\mathrm{P}$ is as in Proposition $\ref{prop:groundwaterscheme}$, applied with $T = \lceil r \rceil$, say. We also let $U_j^n = U_j^n(v_0, \Delta_x^M, \Delta_t^M)$ be  as defined in Proposition $\ref{prop:groundwaterscheme}$. 

Note that $\nu^M(\{j\}) = \int_{j/M}^{(j+1)/M} v_0(x)dx,$ and by Proposition $\ref{prop:ground}$, for all $n\in\N$, and $j\in\Z$ we have that $U_j^n = M \cdot \bP\{G_n = j\}$. Let $v^M: \R \times (0, \infty) \rightarrow \R$ be the function which takes the value $U_j^n$ on $[j/M, (j+1)/M) \times [n/M^3, (n+1)/M^3)$ for $n\in \N$, $j\in \Z$. Then for $0 \le a < b \le 1$,
\[\int\limits_{\ell}^{r}\int\limits_{(a-1/2)(36(t+\varepsilon))^{1/3}}^{(b-1/2)(36(t+\varepsilon))^{1/3}}v^M(x,t)dxdt = \int\limits_{\ell}^{r}\int\limits_{(a-1/2)(36(t+\varepsilon))^{1/3}}^{(b-1/2)(36(t+\varepsilon))^{1/3}}M \cdot \bP\{G_{\lfloor tM^3\rfloor} = \lfloor M x\rfloor \} dxdt.\]
As in the proof of Proposition $\ref{prop:groundwaterscheme}$ it follows that
\begin{align}\label{ground2}&\int\limits_{\ell}^{r}\int\limits_{\left(a-\frac{1}{2}\right)(36(t+\varepsilon))^{1/3}}^{\left(b-\frac{1}{2}\right)(36(t+\varepsilon))^{1/3}} v^M(x,t) dxdt \ge \int\limits_{\ell}^{r} \bP\left\{G_{\lfloor tM^3\rfloor}\in (\lfloor M(a-1/2)c\rfloor, \lfloor M(b-1/2)c\rfloor)\right\}dt ,\end{align}
and 
\begin{equation}\label{ground3}
\int\limits_{\ell}^{r}\int\limits_{\left(a-\frac{1}{2}\right)(36(t+\varepsilon))^{1/3}}^{\left(b-\frac{1}{2}\right)(36(t+\varepsilon))^{1/3}}v^M(x,t) dxdt \le \int\limits_{\ell}^{r} \bP\left\{G_{\lfloor tM^3\rfloor}\in (\lfloor M(a-1/2)c\rfloor, \lfloor M(b-1/2)c\rfloor]\right\}dt.
\end{equation}
By Proposition \ref{prop:groundwaterscheme} $v_\mathrm{P}$ is the BV entropy weak solution to the initial value problem $\partial_t v - \partial_{xx}v^2/2 = 0$ with initial condition $v_0$, by Theorem \ref{thm:evje_karlsen},
\begin{align*}
\lim_{M\rightarrow \infty} \int\limits_{\ell}^{r} \int\limits_{\left(a-\frac{1}{2}\right)(36(t+\varepsilon))^{1/3}}^{\left(b-\frac{1}{2}\right)(36(t+\varepsilon))^{1/3}} v^M(x,t)dxdt &=\int\limits_{\ell}^{r} \int\limits_{\left(a-\frac{1}{2}\right)(36(t+\varepsilon))^{1/3}}^{\left(b-\frac{1}{2}\right)(36(t+\varepsilon))^{1/3}}  v(x,t)dxdt \\
&= \int\limits_{\ell}^{r} 3(b^2 - a^2)- 2(b^3 - a^3) dt\\
& = (r-l)(3(b^2-a^2) - 2(b^3-a^3)).\end{align*}

Combining this with $(\ref{ground1})$, $(\ref{ground2})$, $(\ref{ground3})$, we obtain that for all $b\in (0,1)$,
\[\lim_{M \rightarrow \infty} \bP\left\{\frac{G_{\lfloor WM^3\rfloor} }{(36(W+\varepsilon))^{1/3} M} \in [a,b]\right\} = 3(b^2-a^2) - 2(b^3-a^3) = \int_a^b 6x(1-x)dx.\]
Since the density for a Beta$(2,2)$ random variable is $6x(1-x)\I{x\in[0,1]}$, the result follows.
\end{proof}

\section{\bf Proofs of Theorems~\ref{thm:main1} and~\ref{thm:main2}.}\label{sec:proofs}
In order to strengthen Propositions~\ref{prop:traffic_timeaverage} and~\ref{prop:groundwater_timeaverage} to remove the time averaging, we shall use the following coupling lemmas. 
\begin{lem}\label{stochdom_talshrw}
Fix two probability distributions $\mu$, $\nu$ on $\Z$ and a coupling $(X, Y)$ of $\mu$ and $\nu$, and write $\alpha=\bP(X>Y)$. Fix $k \ge 1$, let $\mu_k$ be the law of $B_k^{\vec{z}}(\emptyset)$ under $\bP_{\mu}$ and let $\nu_k$ be the law of $B_k^{\vec{z}}(\emptyset)$ under $\bP_{\nu}$. Then there exists a coupling $(X', Y')$ of $\mu_k$, $\nu_k$ such that $\bP(X'>Y')= \alpha$. 
\end{lem}
\begin{lem}\label{stochdom}
Fix two probability distributions $\mu$, $\nu$ on $\Z$ and a coupling $(X, Y)$ of $\mu$ and $\nu$, and write $\alpha=\bP(X>Y)$. Fix $k \ge 1$, let $\mu_k$ be the law of $G_k^{\vec{z}}(\emptyset)$ under $\bP_{\mu}$ and let $\nu_k$ be the law of $G_k^{\vec{z}}(\emptyset)$ under $\bP_{\nu}$. Then there exists a coupling $(X', Y')$ of $\mu_k$, $\nu_k$ such that $\bP(X'>Y')\le \alpha$. 
\end{lem}
Both lemmas are proved by the explicit construction of a coupling with the claimed property. In Appendix~\ref{bv_proofs} we prove~Lemma~\ref{stochdom} in detail, then briefly explain how to modify the construction to prove Lemma~\ref{stochdom_talshrw}, since the constructions are nearly identical. 
\begin{proof}[Proof of Theorem \ref{thm:main2}.]
We aim to prove that for any field of IID random variables $\vec{Z}=(Z_v,v \in \cT)$, we have 
\[
(36n)^{-1/3} G_n^{\vec{Z}}(\emptyset) + 1/2 \stackrel{\mathrm{d}}{\longrightarrow} G_\infty\, ,
\]
where, here and later in the proof, $G_\infty$ denotes a $\mathrm{Beta}(2,2)$-distributed random variable. 
We first handle the case that $\vec{Z}=\vec{0}$, or equivalently that the random variables $Z_v$ are $\delta_0$-distributed. At the end of the proof we explain how to extend from this case to general input distributions.

Fix $\varepsilon\in(0,1)$ and let $U$ be a $\mathrm{Uniform}[1,1+\varepsilon]$ random variable, independent of all other randomness in the system.
We recall the definition of $\nu^M$ from Proposition \ref{prop:groundwater_timeaverage}: for $M>0$, $\nu^M$ is the probability measure on $\Z$ such that for all $j\in \Z$
\[\nu^M(\{j\}) = \int_{j/M}^{(j+1)/M} \max\left(\frac{3}{4}\left(\left(\frac{2}{9\varepsilon}\right)^{1/3} - \left(\frac{2x^2}{9\varepsilon}\right)\right),0\right)dx.\]

Now for $M>0$ such that $M^3\in\N$, let $\pi^M$ be the law of $G_{\lfloor UM^3\rfloor - M^3}:=G_{\lfloor UM^3\rfloor - M^3}^{\vec{z}}(\emptyset)$ under $\bP_{\nu^M}$; for $v \in \cL_{M^3}$,  this is also the law of $G_{\lfloor UM^3\rfloor}^{\vec{z}}(v)$ under $\bP_{\nu^M}$. We will use the fact that the law of $G_{\lfloor UM^3\rfloor}=G_{\lfloor UM^3\rfloor}(\emptyset)$ under $\bP_{\nu^M}$ is the same as the law of $G_{M^3}=G_{M^3}(\emptyset)$ under $\bP_{\pi^M}$; see Figure~\ref{fig:coupling}. 
\begin{figure}[htb]
    \centering
    \includegraphics[width=0.7\textwidth]{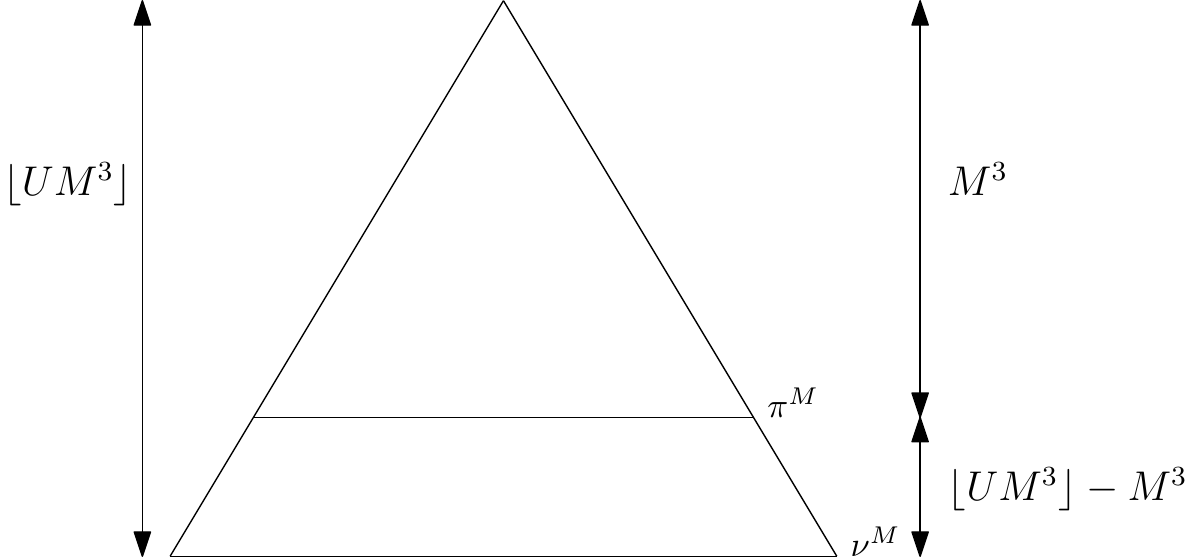}
    \caption{ If the level-$\lfloor UM^3 \rfloor$ inputs are $\nu^M$-distributed then the resulting level-$M^3$ outputs are $\pi^M$-distributed. In other words, for all $\ell \in \N$ and $v \in \cL_{M^3}$ we have $\psub{\nu^M}{G_{\lfloor UM^3\rfloor}^{\vec{z}}(v)=\ell} = \pi^M\{\ell\} = \psub{\pi^M}{x_v=\ell}$, and so also $\psub{\nu^M}{G_{\lfloor UM^3\rfloor}^{\vec{z}}(\emptyset)=\ell} = \psub{\pi^M}{G_{M^3}^{\vec{z}}(\emptyset)=\ell}$.}
    \label{fig:coupling}
\end{figure}

Since $U-1$ is $\mathrm{Uniform}[0,\varepsilon]$, by Proposition \ref{prop:groundwater_timeaverage} applied with $\ell=0$, $r=\varepsilon$ we have that 
\begin{equation}\label{convergence}
    \frac{G_{\lfloor UM^3\rfloor - M^3}}{(36(U-1 + \varepsilon))^{1/3}M} + \frac{1}{2} \convdist G_\infty
\end{equation}
as $M\rightarrow\infty$ along values with $M^3 \in \N$.
Therefore 
\begin{align*}
    \bP_{\nu^M}\left\{G_{\lfloor UM^3\rfloor-M^3} > \frac{(36\cdot 2\varepsilon)^{1/3}M}{2}\right\} &\le \bP_{\nu^M}\left\{G_{\lfloor UM^3\rfloor-M^3} > \frac{(36 (U-1+\varepsilon))^{1/3}M}{2}\right\}\\
    &= \bP_{\nu^M} \left\{\frac{G_{\lfloor UM^3\rfloor - M^3}}{(36(U-1+\varepsilon))^{1/3}M} + \frac{1}{2} > 1 \right\}.
\end{align*}
By (\ref{convergence}), the final probability tends to 0 as $M \rightarrow \infty$, so
\[
\bP_{\nu^M}\left\{G_{\lfloor UM^3\rfloor-M^3} > \frac{(36\cdot 2\varepsilon)^{1/3}M}{2}\right\} \rightarrow 0 \text{ as } M\rightarrow \infty,
\]
and by a similar analysis one finds that 
\[
\bP_{\nu^M}\left\{G_{\lfloor UM^3\rfloor - M^3} < - \frac{(36\cdot 2\varepsilon)^{1/3}M}{2}\right\} \rightarrow 0 \text{ as } M\rightarrow \infty.
\]
Fixing $\alpha >0$, we can therefore choose $M_0$ large enough that for $M \ge M_0$, 
\begin{equation}\label{eq:pos_ineq}
\bP_{\nu^M}\left\{G_{\lfloor UM^3 \rfloor - M^3} > \frac{(36\cdot 2\varepsilon)^{1/3} M}{2}\right\} \le \alpha,
\end{equation}
and 
\[\bP_{\nu^M}\left\{G_{\lfloor UM^3 \rfloor - M^3} < -\frac{(36\cdot 2\varepsilon)^{1/3} M}{2}\right\} \le \alpha.\]
Using the definition of $\pi^M$, the second inequality implies that for $M \ge M_0$, 
under $\bP_{\pi^M}$, the inputs 
$(z_v,v \in \cL_{M^3})=:(G_{M^3}(v), v\in \cL_{M^3})$ are such that
\begin{equation}\label{eq:pi_input_bd}
    \psub{\pi^M}{x_v <- \frac{(36\cdot 2\varepsilon)^{1/3}M}{2}} = \bP_{\nu^M}\left\{G_{\lfloor UM^3\rfloor - M^3}<-\frac{(36\cdot 2\varepsilon)^{1/3}M}{2}\right\} \le \alpha
\end{equation}
Note that we may view (\ref{eq:pi_input_bd}) as stating that that there exists a coupling $(X,Y)$ of $\delta_0$, a Dirac mass at $0$, and $\pi_M$, such that $\p{Y+\frac{(36\cdot 2\varepsilon)^{1/3}M}{2}<X } \le \alpha$.  
We can then apply Lemma \ref{stochdom} to find a coupling $(X_M,Y_M)$ of $G_{M^3}$ under $\bP_{\delta_0}$ and $G_{M^3}$ under $\bP_{\pi^M}$ such that 
\[\bP\left\{Y_M + \frac{(36\cdot 2\varepsilon)^{1/3}M}{2} < X_M\right\} \le \alpha.\]
For all $x\in \R$, this gives
\begin{align*}
    &\limsup_{M\rightarrow \infty}\bP_{\delta_0}\left\{\frac{G_{M^3}}{(36)^{1/3}M} + \frac{1}{2} > x \right\} \\
    &= \limsup_{M\rightarrow \infty} \bP\left\{X_M > (36)^{1/3}M\left(x - \frac{1}{2}\right)\right\}\\
    &\le \alpha + \limsup_{M\rightarrow\infty} \bP\left\{ Y_M + \frac{(36\cdot 2\varepsilon)^{1/3}M}{2} \geq X_M > (36)^{1/3}M\left(x - \frac{1}{2}\right) \right\}\\
    &\le \alpha + \limsup_{M\rightarrow \infty} \bP\left\{Y_M + \frac{(36\cdot 2\varepsilon)^{1/3}M}{2} > (36)^{1/3}M\left(x - \frac{1}{2}\right) \right\}\\
    &= \alpha + \limsup_{M\rightarrow\infty} \bP\left\{\frac{Y_M}{(36)^{1/3}M} + \frac{1}{2} >  x - \frac{(2\varepsilon)^{1/3}}{2} \right\}\\
    &= \alpha + \limsup_{M\rightarrow \infty}\bP_{\pi^M}\left\{\frac{G_{M^3}}{(36)^{1/3}M} + \frac{1}{2} > x - \frac{(2\varepsilon)^{1/3}}{2}\right\}\\
    &= \alpha + \limsup_{M\rightarrow \infty}\bP_{\nu^M}\left\{\frac{G_{\lfloor UM^3\rfloor}}{(36)^{1/3}M} + \frac{1}{2} > x - \frac{(2\varepsilon)^{1/3}}{2}\right\}.
\end{align*}
Combining this result with the fact that $(1+2\varepsilon)/(U+\varepsilon) > 1$, we get that for all $x \in \R$,
\begin{align}\label{liminf}
& \liminf_{M\rightarrow\infty} \bP_{\delta_0}\left\{\frac{G_{M^3}}{(36)^{1/3}M} + \frac{1}{2} \le x \right\} \nonumber\\
& \ge \liminf_{M\rightarrow\infty} \bP_{\nu^M}\left\{\frac{G_{\lfloor UM^3\rfloor }}{(36)^{1/3}M} + \frac{1}{2} \le x - \frac{(2\varepsilon)^{1/3}}{2}\right\} - \alpha, \notag\\
& \ge \liminf_{M\rightarrow\infty} \bP_{\nu^M}\left\{\left(1 + 2\varepsilon\right)^{1/3}\left(\frac{G_{\lfloor UM^3\rfloor }}{(36(U+\varepsilon))^{1/3}M} + \frac{1}{2} \right)\le x - \frac{(2\varepsilon)^{1/3}}{2}\right\} - \alpha \notag\\
&= \bP\left\{G_\infty \le \frac{1}{(1+2\varepsilon)^{1/3}}\left(x - \frac{(2\varepsilon)^{1/3}}{2}\right)\right\} - \alpha\, ;
\end{align}
the last equality follows from Proposition \ref{prop:groundwater_timeaverage}.

Similarly, we may view (\ref{eq:pos_ineq}) as stating that for $M$ sufficiently large there exists a coupling $(X,Y)$ with $X$ having distribution $\pi^M$ and $Y$ having distribution $\delta_0$, such that $\p{X - \frac{(36\cdot 2\varepsilon)^{1/3}M}{2} > Y} \le \alpha$. For $M$ large we may thus apply  Lemma \ref{stochdom}
to find a coupling $(X_M,Y_M)$ of $G_{M^3}$ under $\bP_{\pi^M}$ and $G_{M^3}$ under $\bP_{\delta_0}$ such that 
\begin{equation}\label{eq:second_coupling}
\bP\left\{X_M - \frac{(36\cdot 2\varepsilon)^{1/3}M}{2} > Y_M\right\} \le \alpha.
\end{equation}
This is a different coupling from the one used just above, but we allow ourselves to recycle the notation $(X_M,Y_M)$ as the previous coupling plays no further role. (Note that the marginals of the coupling have switched places.) 

It follows from (\ref{eq:second_coupling}) that for all $x\in\R$,
\begin{align*}
    & \limsup_{M\rightarrow\infty} \bP_{\delta_0}\left\{\frac{G_{M^3}}{(36)^{1/3}M} + \frac{1}{2} \le x \right\} \\
    &= \limsup_{M\rightarrow\infty}\bP\left\{Y_M \le (36)^{1/3}M\left(x - \frac{1}{2}\right)\right\}\\
    &\le \alpha + \limsup_{M\rightarrow\infty}\bP\left\{ X_M - \frac{(36\cdot 2\varepsilon)^{1/3} M}{2} < Y_M \leq (36)^{1/3}M\left(x - \frac{1}{2}\right) \right\}\\
    &\le \alpha + \limsup_{M\rightarrow\infty}\bP\left\{X_M \le (36)^{1/3}M\left(x - \frac{1}{2}\right) + \frac{(36\cdot 2\varepsilon)^{1/3} M}{2}\right\}\\
    &= \alpha + \limsup_{M\rightarrow\infty}\bP_{\pi^M}\left\{\frac{G_{M^3}}{(36)^{1/3}M} + \frac{1}{2} \le x + \frac{(2\varepsilon)^{1/3}}{2}\right\}\\
    &\le \alpha + \limsup_{M\rightarrow\infty}\bP_{\nu^M}\left\{\frac{G_{\lfloor UM^3\rfloor}}{(36(U + \varepsilon))^{1/3}M} + \frac{1}{2} \le x + \frac{(2\varepsilon)^{1/3}}{2}\right\}\\
    &= \alpha + \bP\left\{G_\infty \le x + \frac{(2\varepsilon)^{1/3}}{2}\right\}\, 
\end{align*}
where the final inequality holds since $U+\eps > 1$, and the last equality again holds by Proposition \ref{prop:groundwater_timeaverage}.  Using this  in combination with (\ref{liminf}) gives 
\begin{align*}
    \bP\left\{G_\infty \le \frac{1}{(1+2\varepsilon)^{1/3}}\left(x - \frac{(2\varepsilon)^{1/3}}{2}\right)\right\} - \alpha &\le \liminf_{M\rightarrow\infty} \bP_{\delta_0}\left\{\frac{G_{M^3}}{(36)^{1/3}M} + \frac{1}{2} \le x \right\} \\
    &\le \limsup_{M\rightarrow\infty} \bP_{\delta_0}\left\{\frac{G_{M^3}}{(36)^{1/3}M} + \frac{1}{2} \le x \right\}\\
    &\le \alpha + \bP\left\{G_\infty \le x + \frac{(2\varepsilon)^{1/3}}{2}\right\},
\end{align*}
for all $x\in \R$. Since $\varepsilon \in (0,1)$ was arbitrary we can let $\varepsilon \rightarrow 0$ to obtain that for all $x\in\R$
\begin{align*}
  \bP\{G_\infty \le x \} - \alpha & \le \liminf_{M\rightarrow\infty} \bP_{\delta_0}\left\{\frac{G_{M^3}}{(36)^{1/3}M} + \frac{1}{2} \le x \right\} \\
  & \le \limsup_{M\rightarrow\infty} \bP_{\delta_0}\left\{\frac{G_{M^3}}{(36)^{1/3}M} + \frac{1}{2} \le x \right\}\le \alpha + \bP\left\{G_\infty \le x\right\}.  
\end{align*}
Since $\alpha>0$ was also arbitrary, we may take $\alpha\rightarrow 0$ to get that under $\bP_{\delta_0}$, as $M\rightarrow \infty$ along $M^3\in \N$,
\[\frac{G_{M^3}}{(36)^{1/3}M} + \frac{1}{2} \convdist G_\infty\, .\]

This handles the case that $\vec{Z}=\vec{0}$; we finish the proof by explaining how to extend to general input distributions. It's useful to first note that in any case where all inputs take the same value, the result follows immediately from the case of all-zero inputs, since shifting all quantities in the process by a fixed finite value does not affect the distributional convergence. 

Now suppose the entries of $\vec{Z}$ are \textsc{iid} with some common law $\xi$. Fix $\beta >0$ and let $K=K(\beta)$ be large enough that $\xi([-K,K])>1-\beta$. Then for all $v \in \cT$ we have $\p{Z_v > K} < \beta$, so by Lemma~\ref{stochdom}, for all $x \in \R$, 
\begin{align*}
\limsup_{M \to \infty} \p{\frac{G_{M^3}^{\vec{Z}}}{(36)^{1/3}M} + \frac{1}{2} \le x}
& = \limsup_{M \to \infty} \psub{\xi}{\frac{G_{M^3}}{(36)^{1/3}M} + \frac{1}{2} \le x}\\
& <\limsup_{M \to \infty} \psub{\delta_K}{\frac{G_{M^3}}{(36)^{1/3}M} + \frac{1}{2} \le x}+\beta \\
& = \bP\left\{G_\infty \le x\right\}+\beta\, ,
\end{align*}
the last equality holding since we already established distributional convergence for constant input. It likewise follows that 
\begin{align*}
\liminf_{M \to \infty} \p{\frac{G_{M^3}^{\vec{Z}}}{(36)^{1/3}M} + \frac{1}{2} \le x}
& = \liminf_{M \to \infty} \psub{\xi}{\frac{G_{M^3}}{(36)^{1/3}M} + \frac{1}{2} \le x}\\
& > \liminf_{M \to \infty} \psub{\delta_{-K}}{\frac{G_{M^3}}{(36)^{1/3}M} + \frac{1}{2} \le x}-\beta \\
& = \bP\left\{G_\infty \le x\right\}-\beta\, ;
\end{align*}
combining the two preceding displays and taking $\beta \to 0$, the result follows. 
\end{proof}
\begin{proof}[Proof of Theorem \ref{thm:main1}]
We aim to prove that for any field of IID random variables $\vec{Z}=(Z_v,v \in \cT)$, we have 
\[
(4qn)^{-1/2} B_n^{\vec{Z}}(\emptyset) \convdist B_\infty\, ,
\]
where, here and later in the proof, $B_\infty$ denotes a $\mathrm{Beta}(2,1)$-distributed random variable. We restrict our attention to the case that $\vec{Z}=\vec{0}$; the extension to general input distributions proceeds exactly as in the proof of Theorem~\ref{thm:main2}, using Lemma~\ref{stochdom_talshrw} in place of Lemma~\ref{stochdom}.

Fix $\varepsilon\in (0,1)$, and recall the definition of $\mu^M$ from Proposition \ref{prop:traffic_timeaverage}: for $M > 0$, $\mu^M$ is the probability measure on $\mathbb{Z}$ such that for all $j\in\mathbb{Z}$
\[\mu^M(\{j\}) = \int_{j/M}^{(j+1)/M} \frac{x}{2q\varepsilon}\I{x\in[0,\sqrt{4q\varepsilon}]}dx.\]
Next, let $U$ be a $\mathrm{Uniform}[1,1+\varepsilon]$ random variable, and for $M>0$ such that $M^2\in\mathbb{N}$, let $\pi^M$ be the law of $B_{\lfloor UM^2\rfloor - M^2}$ under $\mathbf{P}_{\mu^M}$; this is also the law of $B_{\lfloor UM^2\rfloor}^{\vec{Z}}(v)$ under $\mathbf{P}_{\mu^M}$ for nodes $v \in \cL_{M^2}$. 
Also, the law of $B_{\lfloor UM^2\rfloor}$ under $\mathbf{P}_{\mu^M}$ is the same as the law of $B_{M^2}$ under $\mathbf{P}_{\pi^M}$. (We have used the shorthand $B_n=B_n^{\vec{Z}}(\emptyset)$ repeatedly in this paragraph.) 

Since $U-1$ is $\mathrm{Uniform}[0,\varepsilon]$, by Proposition \ref{prop:traffic_timeaverage} applied with $\ell = 0$, $r = \varepsilon$ we have that 
\begin{equation}\label{traffic_convergence}\frac{B_{\lfloor UM^2\rfloor - M^2}}{(4q(U-1+\varepsilon))^{1/2}M} \stackrel{d}{\longrightarrow} B,\end{equation}
as $M\rightarrow \infty$ along values with $M^2\in\mathbb{N}$ where $B$ is a $\mathrm{Beta}(2,1)$ random variable. 
Therefore 
\begin{align*}
\bP_{\mu^M}\left\{B_{\lfloor UM^2\rfloor - M^2} > (4q\cdot (2\varepsilon))^{1/2}M\right\} &\le \bP_{\mu^M}\left\{B_{\lfloor UM^2 \rfloor - M^2} >  (4q(U-1 + \varepsilon))^{1/2}M\right\}\\
&= \bP_{\mu^M}\left\{\frac{B_{\lfloor UM^2\rfloor - M^2}}{(4q(U-1 + \varepsilon))^{1/2}M} > 1\right\}.
\end{align*}
By (\ref{traffic_convergence}), the final probability tends to $0$ as $M\rightarrow \infty$ along values with $M^2 \in \mathbb{N}$, so 
\begin{equation}\label{traffic_pos}
\bP_{\mu^M}\left\{B_{\lfloor UM^2\rfloor - M^2} > (4q\cdot (2\varepsilon))^{1/2}M\right\} \rightarrow 0 \text{ as } M\rightarrow\infty.\end{equation}
Therefore, for any $\alpha > 0$, we can choose $M_0$ large enough such that for $M \ge M_0$, 
\[\bP_{\mu^M} \left\{B_{\lfloor UM^2 \rfloor -M^2} > (4q\cdot (2\varepsilon))^{1/2}M\right\} \le \alpha.\]
Also, since the dynamics are monotone non-decreasing, for all $M$ we have 
\begin{equation}\label{traffic_neg}
\psub{\mu^M}{B_{\lfloor UM^2\rfloor - M^2} < 0}=0.\end{equation}

Under $\bP_{\pi^M}$, for $M \ge M_0$ the inputs $(z_v,v \in \cL_{M^2})=:(B_{M^2}(v), v\in \mathcal{L}_{M^2})$ are such that 
\begin{equation}\label{traffic_posineq}\bP_{\pi^M}\left\{z_v > (4q\cdot (2\varepsilon))^{1/2}M\right\} = \bP_{\mu^M}\left\{B_{\lfloor UM^2\rfloor -M^2} > (4q\cdot (2\varepsilon))^{1/2}M\right\} \le \alpha, \end{equation}
and 
\begin{equation}\label{traffic_negineq}\bP_{\pi^M}\left\{z_v < 0 \right\} = \bP_{\mu^M}\left\{B_{\lfloor UM^2\rfloor -M^2} < 0\right\} =0. \end{equation}

By (\ref{traffic_posineq}) we can apply Lemma \ref{stochdom_talshrw} to find a coupling $(X_M,Y_M)$ of $B_{M^2}$ under $\bP_{\pi^M}$ and $B_{M^2}$ under $\bP_{\delta_0}$ such that 
\[\bP\left\{X_M > (4q\cdot(2\varepsilon))^{1/2}M + Y_M\right\} \le \alpha.\]

Then for all $x\in\mathbb{R}$, we obtain the bound 
\begin{align*}
\limsup_{M\rightarrow\infty}\bP_{\delta_0}\left\{\frac{B_{M^2}}{(4q)^{1/2}M} \le x\right\}
&= \limsup_{M\rightarrow\infty} \bP\left\{\frac{Y_M}{(4q)^{1/2}M} \le x \right\}\notag\\
&\le \alpha + \limsup_{M\rightarrow \infty}\bP\left\{\frac{X_M}{(4q)^{1/2}M} - (2\varepsilon)^{1/2} < \frac{Y_M}{(4q)^{1/2}M} \le x\right\}\notag\\
&\le \alpha + \limsup_{M\rightarrow\infty}\bP\left\{\frac{X_M}{(4q)^{1/2}M} < x+(2\varepsilon)^{1/2}\right\}\notag\\
& = \alpha + \limsup_{M\rightarrow\infty}\psub{\pi^M}{\frac{B_{M^2}}{(4q)^{1/2}M} <   x+(2\varepsilon)^{1/2}}\notag\\
& = \alpha + \limsup_{M\rightarrow\infty}\psub{\mu^M}{\frac{B_{\lfloor UM^2\rfloor}}{(4q)^{1/2}M} <  x+ (2\varepsilon)^{1/2}}\notag\, .
\end{align*}
Using that $U+\eps > 1$, then using Proposition \ref{prop:traffic_timeaverage} applied with $\ell = 1$, and $r = 1+\varepsilon$, this gives
\begin{align}\label{traffic_limsup}
\limsup_{M\rightarrow\infty}\bP_{\delta_0}\left\{\frac{B_{M^2}}{(4q)^{1/2}M} \le x\right\}
& \le \alpha + \limsup_{M\rightarrow\infty}\psub{\mu^M}{\frac{B_{\lfloor UM^2\rfloor}}{(4q(U+\eps))^{1/2}M} < x+  (2\varepsilon)^{1/2}}\notag\\
& = \alpha + \bP\left\{B_\infty \le x+ (2\varepsilon)^{1/2} \right\}. 
\end{align}

Likewise, using (\ref{traffic_negineq}) and Lemma \ref{stochdom_talshrw} we may 
find a (different) coupling $(X_M,Y_M)$ of $B_{M^2}$ under $\bP_{\delta_0}$ and $B_{M^2}$ under $\bP_{\pi^M}$ such that 
$\p{X_M>Y_M}=0$. In other words, $B_{M^2}$ under $\bP_{\delta_0}$ is stochastically dominated by $B_{M^2}$ under $\bP_{\pi^M}$. 
It follows that for all $x \in \R$, 
\[
\liminf_{M \to \infty} \psub{\delta_0}{\frac{B_{M^2}}{(4q)^{1/2}M} \le x}  \ge
\liminf_{M \to \infty} \psub{\pi^M}{\frac{B_{M^2}}{(4q)^{1/2}M} \le x}.
\]
Using that $(1+2\eps)/(U+\eps)$ > 1, then using Proposition \ref{prop:traffic_timeaverage} applied with $\ell = 1$, and $r = 1+\varepsilon$, we get that for all $x \in \R$, 
\begin{align*}
\liminf_{M \to \infty} \psub{\delta_0}{\frac{B_{M^2}}{(4q)^{1/2}M} \le x}
&
\ge 
\liminf_{M \to \infty} \psub{\pi^M}{(1+2\eps)^{1/2}\frac{B_{M^2}}{(4q(U+\eps))^{1/2}M} \le x} \\
& = 
\liminf_{M \to \infty} \psub{\mu^M}{(1+2\eps)^{1/2}\frac{B_{\lfloor UM^2\rfloor}}{(4q(U+\eps))^{1/2}M} \le x} \\
& = 
\p{(1+2\eps)^{1/2} B_{\infty} \le x}\, .
\end{align*}
Combining this with (\ref{traffic_limsup}) gives 
\begin{align*}
\p{(1+2\eps)^{1/2} B_{\infty} \le x} & \le \liminf_{M \to \infty} \psub{\delta_0}{\frac{B_{M^2}}{(4q)^{1/2}M} \le x}\\
& 
\le \limsup_{M \to \infty} \psub{\delta_0}{\frac{B_{M^2}}{(4q)^{1/2}M} \le x}
\le \alpha + \bP\left\{B_\infty \le x+ (2\varepsilon)^{1/2} \right\}. 
\end{align*}
Since $\alpha>0$ and $\eps \in (0,1)$ were arbitrary, it follows that $((4q)^{1/2}M)^{-1}B_{M^2}^{\vec{0}} \convdist B_\infty$, as required. 
\end{proof}
\section{\bf Conclusion}\label{sec:conc}
There are several natural avenues for extensions of our results which are deserving of study. The first three points below relate specifically to hipster random walks. 
\begin{itemize}
\item The robustness of Theorems~\ref{thm:main1} and~\ref{thm:main2} with respect to the law of the inputs is due to the fact that, by the coupling lemmas, changes to the input law have an essentially additive effect on the dynamics, and this effect vanishes after rescaling. 

We can say less about robustness with respect to changes in the step distribution. 
In general, for a hipster random walk with bounded steps $(D_v,v \in \cT)$, we would expect that if the steps are centred then one should expect a version of Theorem~\ref{thm:main2} to hold (with a normalizing constant depending on the step distribution), whereas if the steps have non-zero mean then a version of Theorem~\ref{thm:main1} should hold. 

As a special case of the second assertion, one might try to extend Theorem~\ref{thm:main1} to hipster random walks with non-negative, bounded integer steps. If the steps take values in $\{0,1,\ldots,M\}$ with 
\[
\p{D_v=i} = c_i, 
\]
then one would obtain the recurrence 
\[
r^{n+1}_k = r^n_k(1-r^n_k) + \sum_{i=0}^M c_i (r^n_{k-i})^2. 
\]

Another natural special case to consider is that of asymmetric simple random walk, where $\p{D_v=1}=q=1-\p{D_v=-1}$, for $q \ne 1/2$. This case may even be accessible with a variant of the techniques of the current work, since the resulting finite difference scheme appears to fit within the framework of \cite{ek}. However, although the resulting \textsc{pde} has the same long-term behaviour as Burgers' equation, at any finite time the solution behaves like a mixture of Burgers' equation and the porous membrane equation, and this seems to complicate the analysis. 

\item For unbounded step distributions -- and in particular for heavy-tailed step distributions -- one should be able to construct more exotic behaviour. It would be quite interesting to understand whether there is a dictionary between the possible behaviours of the hipster random walk and the various solutions of the associated \textsc{pde}s. 
\end{itemize}
We conclude with some further potential research directions, in the spirit of this work but not specifically related to hipster random walks. 
\begin{itemize}
\item Instead of a hipster random walk, one may consider a fomo random walk (fomo stands for ``fear of missing out''). Here the combination rule is 
\[
f_v(x,y) = 	\begin{cases}
				x						& \mbox{ if }x=y \\
				xA_v+y(1-A_v)+D_v	& \mbox{ if }x\ne y\, ,
			\end{cases}
\]
where as before the $A_v$ are Bernoulli$(1/2)$-distributed. 
For this dynamics, walkers are happy when they have company, and only move when they find themselves alone. The recurrence relation for the fomo symmetric simple random walk (when the $D_v$ are centred $\pm 1$ random variables) can be written as 
\[
r^{n+1}_k - r^n_k = 
\frac{1}{2} \left(r^n_{k-1}-2r^n_k+r^n_{k-1}\right)  -\frac{1}{2} \left((r^n_{k-1})^2 - 2(r^n_k)^2+(r^n_{k+1})^2\right)\, ,
\]
which one may suppose converges to a solution of the \textsc{pde} $\partial_t u = (\partial_{xx} u - \partial_{xx}(u^2))/2$. The presence of a diffusive term should make this model's analysis somewhat more straightforward. 

\item More generally, what conditions on a discrete difference equation imply that it can be interpreted as describing the evolution of the distribution function for an integer-valued recursive distributional equation? Conversely, which integer \textsc{RDE}s yield difference equations which may be interpreted as numerical schemes (and fruitfully analyzed using techniques from numerical analysis)? Also: can this approach be of any use in settings where integrality is not preserved (such as that of the random hierarchical lattice)? 

\item This paper imports theorems from numerical analysis to prove probabilistic results; perhaps information can also flow in the other direction. The development of any reasonably general techniques for analyzing such probabilistic systems would seem likely to simultaneously establish new stability/convergence results for numerical approximations of \textsc{pde}s. Thus far, we are not aware of any theorems in the numerical analysis literature which have been proved in such a way. 
\end{itemize}

\appendix

\section{\bf Remaining proofs.} \label{bv_proofs}
In this section we prove Propositions~\ref{prop:trafficscheme} and~\ref{prop:groundwaterscheme}, and Lemmas~\ref{stochdom_talshrw} and~\ref{stochdom}
\begin{proof}[Proof of Proposition~\ref{prop:trafficscheme}]
We first verify monotonicity. Since $K_\mathrm{B}\equiv0$, the function  $S:\R^3 \rightarrow \R$ defined by (\ref{monotone_scheme}) is
\[S(u^-, u, u^+)  = \frac{u}{\Delta_t^M} - \frac{q}{\Delta_x^M}(u^2-(u^-)^2) \\
     = M^2\cdot u - qM\cdot(u^2 - (u^-)^2).\]
The function $S$ is nondecreasing in all of its arguments on $[0,M/(2q)]$, so for $M$ sufficiently large, the approximation scheme $(U^n_j(u_0,\Delta_x^M,\Delta_t^M))_{n \in \N,j \in \Z}$ is monotone on $[0,1/(\sqrt{q \varepsilon})]$.

We now turn to the first claim of the proposition. The function $u=u_B$ is clearly of bounded variation since it has bounded Lipschitz constant in the compact region 
\[\{(x,t): 0 \le x \le \sqrt{4q(t+\varepsilon)},0 \le t \le T\}\, \]
and is zero outside this region. 
Also, by definition $u(\cdot, 0) \equiv u_{0}$ and it is clear that $f(u_0)-K'(u_0)$ is of bounded variation.
To prove the proposition, it then remains to show that 
\begin{equation}\label{eq:mustbepositive}
\int_{\R\times[0,T]} \sgn(u-c) \cdot \Big((u-c) \partial_t \phi + q(u^2-c^2) \partial_x \phi \Big) \mathrm{d}t\mathrm{d}x + 
\int_\R |u_0-c| \phi(x,0)\mathrm{d}x \ge 0\, .
\end{equation}
for all $c \in \R$ and all non-negative $\phi \in C^\infty(\R \times [0,T])$ with compact support such that $\phi|_{t=T}\equiv 0$. 
Before beginning the analysis, note that $\phi$, $\partial_t\phi$, $\partial_x\phi$ and $u$ are all bounded on $\R \times [0,T]$ hence there is no issue when changing the order of integration. The proof naturally splits into four cases according to whether  $c \le 0, \ c\in(0, \sqrt{1/(q(t+\varepsilon))}]$, $ c\in (\sqrt{1/(q(t+\varepsilon))}, \sqrt{1/(q \varepsilon)}]$, and  $c >\sqrt{1/(q\varepsilon)}$; see Figure~\ref{fig:regions}.
\begin{figure}
\includegraphics[width=1\textwidth]{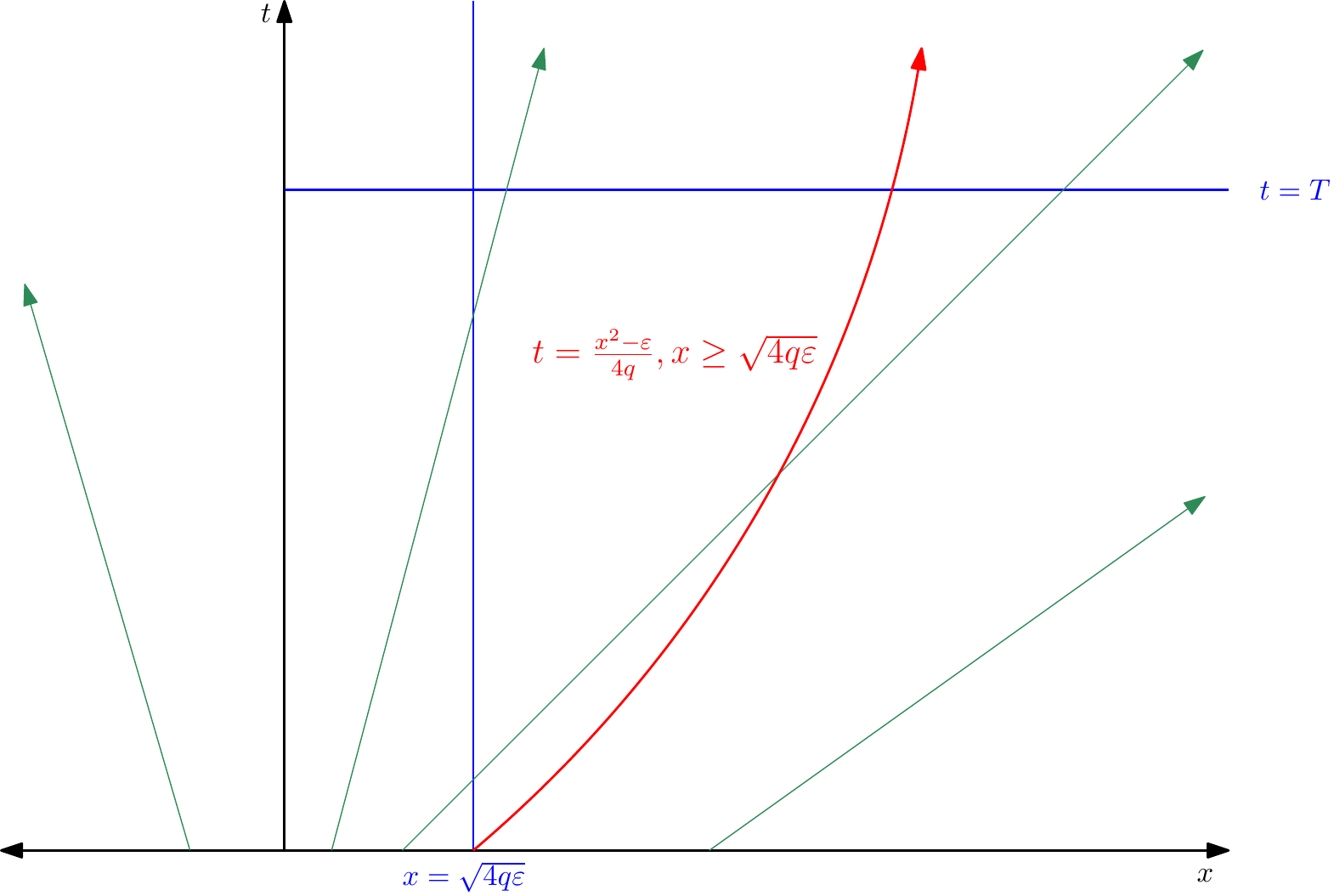}
\caption{The region of integration and different ``sign-change'' regimes for different values of $c$. Read from left to right, the straight green lines represent equations of the form $t=x/(2qc)-\varepsilon$ for $c \le 0$, $c \in (0,1/\sqrt{q(t+\varepsilon)}]$, $c \in (1/\sqrt{q(t+\varepsilon)},1/\sqrt{q\varepsilon}]$ and $c > 1/\sqrt{q\varepsilon}$, respectively.}
\label{fig:regions}
\end{figure}

The most involved case is when $c \in \left(\frac{1}{\sqrt{q(t+\varepsilon)}}, \frac{1}{\sqrt{q\varepsilon}}\right]$. We will provide a full proof only for this case. 
Define the two regions
\[\mathrm{R^-} := \left\{(x,t): x \leq \min(2qc(t+\varepsilon), \sqrt{4q(t+\varepsilon)}), t \in [0, T]\right\},\] and 
\[\mathrm{R^+} := \left\{(x,t): 2qc(t+\varepsilon) \leq x \leq \sqrt{4q(t+\varepsilon)}, t \in [0, T]\right\},\]
and write $n_{\mathrm{R}^-}, \ n_{\mathrm{R}^+}$ for their respective outward normal vectors. 

For a $C^1$ function $F=(F^{(x)},F^{(y)}):\R^2 \to \R^2$ we write $\mathrm{div}(F) = \partial_x F^{(x)} + \partial_y F^{(y)}:\R^2 \to \R$ for the divergence of $F$.
Remark that for $(x,t)$ lying in the interior of either $R^+$ or $R^-$, 
\begin{align*}
        \mathrm{div}\left(q(u^2-c^2)\phi, (u-c)\phi\right) & = \partial_x\left(q(u^2-c^2)\phi\right) + \partial_t((u-c)\phi)\\
        & = \partial_x (q\cdot u^2\cdot \phi) + q(u^2-c^2)\cdot \partial_x\phi + \partial_t u \cdot \phi + (u-c)\cdot\partial_t\phi\\
        & =  q(u^2-c^2)\cdot \partial_x\phi + (u-c)\cdot\partial_t\phi.
\end{align*}

We can therefore rewrite the left hand side of ($\ref{eq:mustbepositive}$) as
\begin{align*}\int_{R^+}\mathrm{div}\left(q(u^2-c^2)\phi, (u-c)\phi\right) - \int_{R^-}\mathrm{div}\left(q(u^2-c^2)\phi, (u-c)\phi\right) + \int_{\R}|u_0-c| \phi(x,0)\mathrm{d}x,\end{align*}
and by applying the divergence theorem, this can in turn be written as 
\begin{align}\label{diventropy}
\int_{\partial \mathrm{R^+}} &\left(q(u^2-c^2)\phi, (u-c)\phi\right)\cdot n_{\mathrm{R^+}} - \int_{\partial \mathrm{R^-}} \left(q(u^2-c^2)\phi, (u-c)\phi\right)\cdot n_{\mathrm{R^-}} + \int_{\R}|u_0 - c|\phi(x,0)dx \notag\\
&\ge\int_{\mathrm{T}^+} \left(q(u^2-c^2)\phi, (u-c)\phi\right)\cdot n_{\mathrm{R}^+} + \int_{\mathrm{L}}\left(q(u^2-c^2)\phi, (u-c)\phi\right)\cdot n_{\mathrm{R}^+}\\&\notag -\left( \int_{\mathrm{T^-}} \left(q(u^2-c^2)\phi, (u-c)\phi \right)\cdot n_{\mathrm{R}^-} + \int_{\mathrm{L}} \left(q(u^2-c^2)\phi, (u-c)\phi \right)\cdot n_{\mathrm{R}^-}\right),
\end{align}
where \[\mathrm{L} := \left\{ (x,t): x = 2qc(t+\varepsilon), 0 \leq t \leq \frac{4q}{c^2}-\varepsilon\right\},\] \[\mathrm{T}^+ := \left\{(x,t) : x = \sqrt{4q(t+\varepsilon)},\  0 \leq t \leq \frac{1}{c^2q} - \varepsilon\right\},\] and \[\mathrm{T}^- := \left\{(x,t) : x = \sqrt{4q(t+\varepsilon)},\  \frac{1}{c^2q} - \varepsilon < t \leq T\right\}.\] 
Note that $u-c \equiv 0$ on $L$, and therefore we can rewrite (\ref{diventropy}) as 
\[\int_{\mathrm{T}^+} \left(q(u^2-c^2)\phi, (u-c)\phi\right)\cdot n_{\mathrm{R}^+} - \int_{\mathrm{T^-}} \left(q(u^2-c^2)\phi, (u-c)\phi \right)\cdot n_{\mathrm{R}^-} .\]

On $T^+$ and $T^-$ we have $n_{\mathrm{R^+}} = n_{\mathrm{R^-}}$, where $n_{\mathrm{R}^+} = \left( \mid \frac{(t+\varepsilon)}{t+\varepsilon + q}\mid ^\frac{1}{2},\ -\mid \frac{q}{t+\varepsilon + q}\mid ^\frac{1}{2}\right)$. This yields
\begin{align*}\label{lowerbdEntropy2}
     \int_{\mathrm{T}^+} &\left(q(u^2-c^2)\phi, (u-c)\phi\right)\cdot n_{\mathrm{R}^+} - \int_{\mathrm{T}^-} \left(q(u^2-c^2)\phi, (u-c)\phi\right)\cdot n_{\mathrm{R}^-}\\
    & = \int_{\mathrm{T}^+}\frac{c \cdot \phi}{|t+\varepsilon+q|^{1/2}}\left( \sqrt{q} - qc\sqrt{(t+\varepsilon)}\right) - \int_{\mathrm{T}^-}\frac{c \cdot \phi}{|t+\varepsilon+q|^{1/2}}\left(\sqrt{q} - qc\sqrt{(t+\varepsilon)}\right)\\
    &\ge\int_{\mathrm{T}^+}\frac{c \cdot \phi}{|t+\varepsilon+q|^{1/2}}\left( \sqrt{q} - q\sqrt{\frac{1}{q}}\right) - \int_{\mathrm{T}^-}\frac{c \cdot \phi}{|t+\varepsilon+q|^{1/2}}\left(\sqrt{q} - q\sqrt{\frac{1}{q}}\right)\\
    & = 0.
\end{align*}
The final inequality holds because $c\cdot|4(t+t_0) + 4q | ^{-\frac{1}{2}}\phi>0$, on $T^+$ we have $c\sqrt{t+\varepsilon}\leq\sqrt{\frac{1}{q}}$, and on $T^-$, $c\sqrt{t+\varepsilon}\geq\sqrt{\frac{1}{q}}$.  From this result we conclude that inequality ($\ref{eq:mustbepositive}$) holds in the case $c \in \left(\sqrt{\frac{4q}{t+\varepsilon}}, \sqrt{\frac{4q}{\varepsilon}}\right]$.

As inequality (\ref{eq:mustbepositive}) is satisfied for all cases of $c\in\R$, we conclude that $u_\mathrm{B}$ is the BV entropy weak solution to (\ref{eq:convection_diffusion}) with $f = f_\mathrm{B}$, $K = K_\mathrm{B}$ and $u_0 = u_{\mathrm{B}}(\cdot, 0)$.
\end{proof}
\smallskip
\begin{proof}[Proof of Proposition~\ref{prop:groundwaterscheme}]
Again, we begin by verifying monotonicity. Since $f_{\mathrm{P}} \equiv 0$, the function $S: \R^3 \rightarrow \R$ defined by $(\ref{monotone_scheme})$ is 
\begin{align*}S(v^-, v, v^+) & = \frac{v}{\Delta_t^M} + \frac{1}{2(\Delta_x^M)^2}((v^+)^2 - 2v^2 + (v^-)^2) \\
& = M^3v - \frac{v^2}{2}M^2 + \frac{M^2}{4}((v^+)^2 + (v^-)^2).
\end{align*}
The function $S$ is non-decreasing in all of its arguments on $[0,M]$, so for sufficiently large $M$  the approximation scheme $(U_j^n(v_0, \Delta_x^M, \Delta_t^M))_{n\in\N, j\in\Z}$ is monotone on $[0,(3/4)\left(
2/(9\varepsilon)\right)^{1/3}]$.

We now turn to the first claim of the proposition. Similarly to in the proof of (\ref{prop:trafficscheme}) it is clear that the functions, $v \equiv v_\mathrm{P}$, and $f(v_0) - \partial _x K(v_0)$ with $v(\cdot, 0) \equiv v_0$ are of bounded variation. Therefore, it remains to show that 
\begin{equation} \label{eq:groundmustbepos} 
\int_{\R\times [0,T]} \sgn(v - c) \cdot \left((v-c)\partial_t\phi - vv_x\partial_x\phi\right)dtdx + \int_\R |v_0 - c| dx \ge 0, \end{equation}
for all $c\in\R$ and all non-negative $\phi\in C^\infty(\R\times[0, T])$ with compact support such that $\phi |_{t = T} \equiv 0$. The proof splits into three cases: $c \le 0$, $c\in (0, A(T+\varepsilon)^{-\frac{1}{3}})$, and $c \geq  A(T+\varepsilon)^{-\frac{1}{3}}$ where $A = (\frac{9}{2})^{\frac{2}{3}}\cdot \frac{1}{6}$; see Figure~\ref{fig:regions2}. 
\begin{figure}
    \centering
    \includegraphics[width = \textwidth]{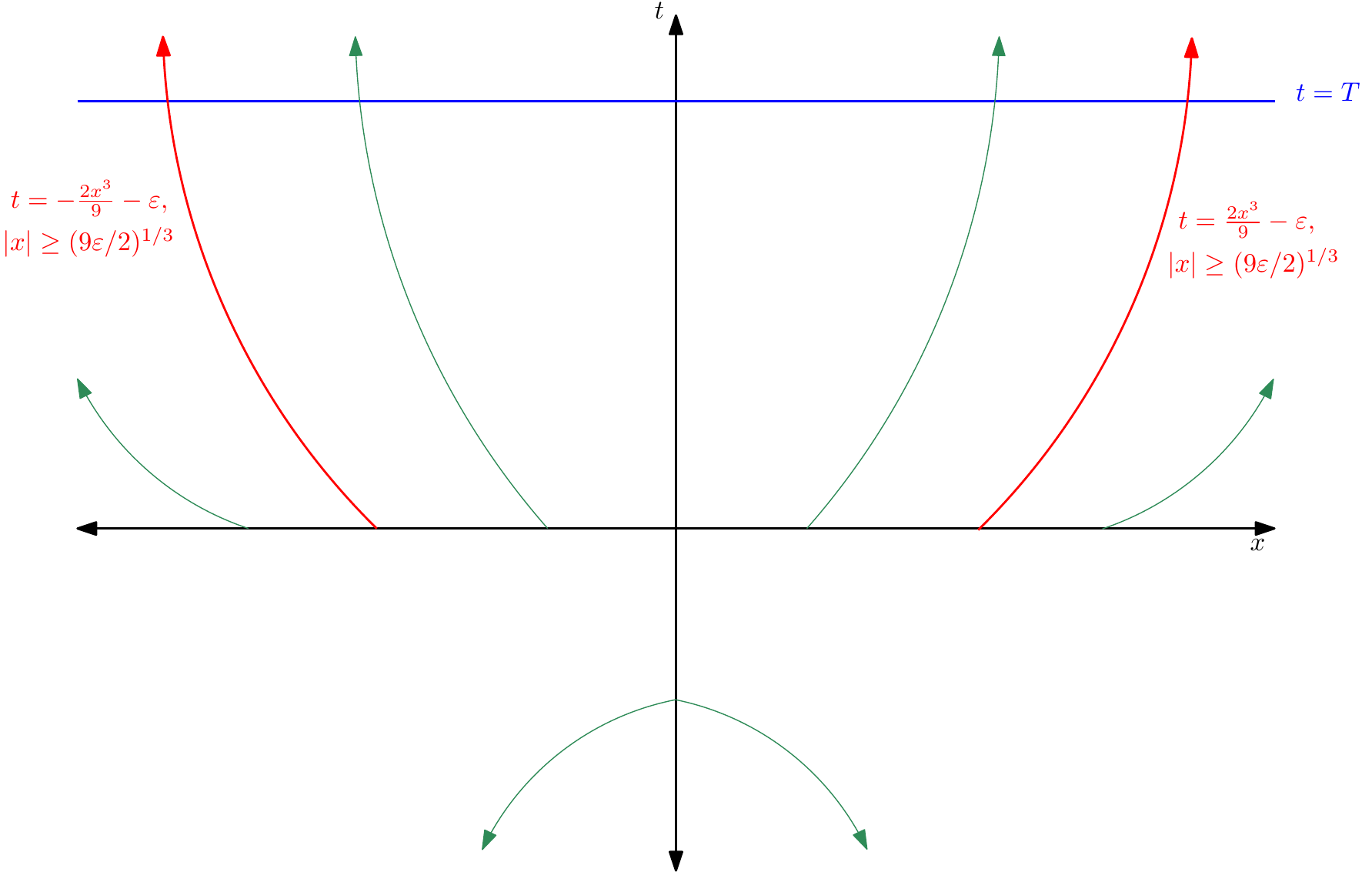}
    \caption{The region of integration and different ``sign-change" regimes for different values of $c$. The green curves represent equations of the form $x = (((9(t+\varepsilon))/2)^{2/3} - 6c(t+\varepsilon))^{1/2}$, where for the outermost curves in the upper quadrants $c < 0$, for the innermost curves in the upper quadrants $c \in (0, A(T+\varepsilon)^{-1/3})$, and for the curve in the lower quadrants $c \ge A(T+\varepsilon)^{-1/3}$.}
    \label{fig:regions2}
\end{figure}
In this setting the most involved case under consideration is when $c\in (0, A/(T+\varepsilon)^{-\frac{1}{3}}]$. We provide a full proof for this case, with the other cases following by similar arguments.

Let  $\phi\in C^\infty(\R\times[0, T])$ be a non-negative function with compact support such that $\phi |_{t = T} \equiv 0$ and define the regions
\[\mathrm{R}^+ := \left\{(x,t)~:~ |x| \le \left(\left(\frac{9(t+\varepsilon)}{2}\right)^{2/3} - 6c(t+\varepsilon)\right)^{1/2}, ~ 0 \le t \le T\right\},\] \[\mathrm{R}^- := \left\{(x,t)~:~  \left(\left(\frac{9(t+\varepsilon)}{2}\right)^{2/3} - 6c(t+\varepsilon)\right)^{1/2} \le |x| \le \left(\frac{9(t+\varepsilon)}{2}\right)^{1/3},~ 0 \le t \le T\right\},\]
 and 
 \[\mathrm{R}_0 := \left\{(x,t)~:~  |x| \ge\left(\frac{9(t+\varepsilon)}{2}\right)^{1/3} ,~ 0 \le t \le T\right\}.\]

Observe for (x,t) in the interior of any of $R^+, R^-$ or $R_0$.
\begin{align*}
\mathrm{div}(- vv_x\phi, (v-c)\phi) &= \partial_x (-vv_x \phi) + \partial_t((v-c) \phi)\\
						&= -(v_x)^2\cdot\phi - vv_{xx}\cdot\phi - vv_x\partial_x\phi + v_t\cdot\phi + (v-c)\partial_t\phi\\
						&= -\left(\partial_{xx} \left(\frac{v^2}{2}\right)\right)\cdot\phi - vv_x \partial_x\phi + v_t \cdot \phi +  (v-c)\partial_t\phi\\
						&= -vv_x\partial_x\phi + (v-c)\partial_t\phi.
\end{align*}

We can therefore rewrite the left hand side of (\ref{eq:groundmustbepos}) as 
\begin{align*}
    \int_{\mathrm{R}^+} \mathrm{div}\left(-vv_x\phi, (v-c)\phi\right) - \int_{\mathrm{R}^-} \mathrm{div}\left(-vv_x\phi, (v-c)\phi\right) \\
    - \int_{\mathrm{R}_0} \mathrm{div}\left(-vv_x\phi, (v-c)\phi\right)
+ \int_\R |v_0 - c|\phi(x,0)dx.
\end{align*}

 Applying the divergence theorem, this can in turn be written as \begin{align}\label{gw:divtheorem}
       \notag\int_{\partial \mathrm{R}^+}\left(-vv_x\phi, (v-c)\phi\right)\cdot n_{\mathrm{R}^+} &- \int_{\partial\mathrm{R}^-}\left(-vv_x\phi, (v-c)\phi\right)\cdot n_{\mathrm{R}^-}\\
      & - \int_{\partial\mathrm{R}_0}\left(-vv_x\phi, (v-c)\phi\right)\cdot n_{\mathrm{R}_0} + \int_\R |v_0 - c | \phi(x,0)dx,
 \end{align}
where $n_{\mathrm{R}^+}$, $n_{\mathrm{R}^-}$ and $n_{\mathrm{R}_0}$ are the outward normal vectors on $\mathrm{R}^+$, $\mathrm{R}^-$, and $\mathrm{R}_0$ respectively. Now, let
 \begin{align*}
  &  L^+ := \left\{(x,t)~:~ x = \left(\left(\frac{9(t+\varepsilon)}{2}\right)^{2/3} - 6c(t+\varepsilon)\right)^{1/2},~ 0 \le t\le T\right\},\\
&L^- := \left\{(x,t)~:~ x = -\left(\left(\frac{9(t+\varepsilon)}{2}\right)^{2/3} - 6c(t+\varepsilon)\right)^{1/2},~ 0 \le t\le T\right\},\\
&T^+ := \left\{(x,t)~:~ x = \left(\frac{9(t+\varepsilon)}{2}\right)^{1/3},~ 0 \le t\le T\right\},\\
&T^- :=\left \{(x,t)~:~ x = -\left(\frac{9(t+\varepsilon)}{2}\right)^{1/3},~ 0 \le t\le T\right\}.
\end{align*}

Note that on the regions $L^+$ and $L^-$ we have that $v-c \equiv 0$, and on the regions $T^+$ and $T^-$, $v\equiv0$. We can then bound $(\ref{gw:divtheorem})$ from below by
\begin{align*}
   &\left( \int_{\mathrm{L}^+}\left(-vv_x\phi, 0\right)\cdot n_{\mathrm{R}^+} + 
    \int_{\mathrm{L}^-}\left(-vv_x\phi, 0\right)\cdot n_{\mathrm{R}^+}\right) \\
   & -\left( \int_{\mathrm{L}^+}\left(-vv_x\phi, 0\right)\cdot n_{\mathrm{R}^-} + 
    \int_{\mathrm{L}^-}\left(-vv_x\phi, 0\right)\cdot n_{\mathrm{R}^-} 
    + \int_{\mathrm{T}^+}\left(0, -c\phi\right)\cdot n_{\mathrm{R}^-} + \int_{\mathrm{T}^-}\left(0, -c\phi\right)\cdot n_{\mathrm{R}^-}  \right)\\
   &- \left(\int_{\mathrm{T}^+}\left(0, -c\phi\right)\cdot n_{\mathrm{R}_0} + \int_{\mathrm{T}^-}\left(0, -c\phi\right)\cdot n_{\mathrm{R}_0}\right)\\
    & = \left( \int_{\mathrm{L}^+}\left(-vv_x\phi, 0\right)\cdot n_{\mathrm{R}^+} - \int_{\mathrm{L}^+}\left(-vv_x\phi, 0\right)\cdot n_{\mathrm{R}^-}\right)  + \left( \int_{\mathrm{L}^-}\left(-vv_x\phi, 0\right)\cdot n_{\mathrm{R}^+} - \int_{\mathrm{L}^-}\left(-vv_x\phi, 0\right)\cdot n_{\mathrm{R}^-}\right)\\
    & + \left( \int_{\mathrm{T}^+}\left(0, -c\phi\right)\cdot n_{\mathrm{R}^-} - \int_{\mathrm{T}^+}\left(0, -c\phi\right)\cdot n_{\mathrm{R}_0}\right) + \left( \int_{\mathrm{T}^-}\left(0, -c\phi\right)\cdot n_{\mathrm{R}^-} - \int_{\mathrm{T}^-}\left(0, -c\phi\right)\cdot n_{\mathrm{R}_0}\right) \\
    & = 2 \cdot \int_{\mathrm{L}^+}\left(-vv_x\phi, 0\right)\cdot n_{\mathrm{R}^+} + 2 \cdot \int_{\mathrm{L}^-}\left(-vv_x\phi, 0\right)\cdot n_{\mathrm{R}^+}.
\end{align*}
The last equality follows from the fact that 
\[n_{\mathrm{R}^-} =  n_{\mathrm{R}^+}\begin{pmatrix}-1 & 0 \\ 0 & 1\end{pmatrix} \text{ on } L^+,\ L^-\]
and 
\[n_{\mathrm{R}_0} =  n_{\mathrm{R}^-}\begin{pmatrix}-1 & 0 \\ 0 & 1\end{pmatrix} \text{ on } T^+,\ T^-.\]
Direct calculation gives that
  $ n_{\mathrm{R}^+} = \left( B^{-1},\ DB^{-1}\right) \text{ on } \mathrm{L}^+ \text{, and }  n_{\mathrm{R}^+} = \left( -B^{-1},\ DB^{-1}\right) \text{ on } \mathrm{L}^-$, where 
    \[B = \left( \frac{9}{4}\left( \left(\frac{2}{9(t+\varepsilon)}\right)^{1/3} -2c\right)^2 \cdot \left( \left(\frac{9(t+\varepsilon)}{2}\right)^{2/3} - 6c(t+\varepsilon)\right)^{-1}     \right)^{1/2}\]
    and 
    \[D = \frac{3}{2}\left( \left( \frac{2}{9(t+\varepsilon)}\right)^{1/3}-6c \right) \cdot \left( \left(\frac{9(t+\varepsilon)}{2}\right)^{2/3} - 6c(t+\varepsilon)\right)^{-1/2}.\]
We thus have
    $(-vv_x\phi, 0)\cdot n_{R^{+}} 
   = -vv_x\phi B^{-1} \geq 0$ on $\mathrm{L}^+$,
  and  $(-vv_x\phi, 0)\cdot n_{R^{+}} = vv_x\phi B^{-1} \geq 0$ on $\mathrm{L}^-$. Combining the above results, it follows that for all non-negative $\phi\in C^\infty(\R \times [0,T])$ with compact support such that $\phi|_{t=T} \equiv 0$, and for all $c\in (0, 4q(T+\varepsilon)^{-\frac{1}{3}})$, 
\[\int_{\R\times[0,T]} \sgn(v-c)\cdot ((v-c)\partial_t\phi - vv_x\partial_x\phi)dtdx + \int_\R |v_0 - c| dx \ge 0.\]
The same inequality follows for the other ranges of c, and by similar arguments, and we can therefore conclude that $v_\mathrm{P}$ is the BV entropy weak solution to (\ref{eq:convection_diffusion}) with $f = f_\mathrm{P}$, $K = K_{\mathrm{P}}$, and $v_0 = v_{\mathrm{P}}(\cdot, 0)$.
\end{proof}

\begin{proof}[Proof of Lemma \ref{stochdom}.]
By induction, it suffices to prove the lemma when $k=1$, and we now restrict our attention to this setting. 

Let $(A,B)$ and $(C,D)$ be independent of one another, with each pair distributed according to the coupling $(X,Y)$. Then $A$ and $C$ are independent and $\mu$-distributed, and $B$ and $D$ are independent and $\nu$-distributed. We shall couple the symmetric simple hipster random walk dynamics with input $(A,C)$ to those with input $(B,D)$, via a case-by-case construction of the coupling dynamics. 

Define the events $E_1 = \{ A \le B, C \le D\}$, $E_2 = \{A \le B, C > D\}$, $E_3 = \{A > B, C \le D\}$, and $E_4 = \{A > B, C > D\}$. For $i\in\{1,2,3,4\}$ we consider the following sub-cases: 
\begin{center}
\begin{tabular}{c c c}
(i) $E_i \cap \{ A = C\} \cap \{B \neq D\}$ & & (ii) $E_i \cap \{ A \neq C\} \cap \{B \neq D\}$\\
(iii) $E_i \cap \{A \neq C\} \cap \{B = D\}$,& and & (iv) $E_i \cap \{A = C\} \cap \{B = D\}$.
\end{tabular}
\end{center}

We begin by constructing the coupling for $E_1$. In case (i), we construct the coupling $(X', Y')$ as
\[\begin{cases}
Y' = \max (B,D) \Longleftrightarrow X' = A+1\\
Y' = \min (B,D) \Longleftrightarrow X' = A-1,\\
\end{cases}\]
which gives $\bP\{X' > Y' ~|~ E_1 \cap  \{ A = C\} \cap \{B \neq D\}\} = 0$.\\
For case (ii), we construct the coupling as 
\[\begin{cases}
Y' = D \Longleftrightarrow X' = C\\
Y' = B \Longleftrightarrow X' = A,\\
\end{cases}\]
giving that $\bP\{X' > Y' ~|~ E_1 \cap  \{ A \neq C\} \cap \{B \neq D\}\} = 0$.\\
Further, for case (iii) the coupling is 
\[\begin{cases}
Y' = D+1 \Longleftrightarrow X' = \max(A,C)\\
Y' = D-1 \Longleftrightarrow X' = \min(A,C),\\
\end{cases}\]
which again gives $\bP\{X' > Y' ~|~ E_1 \cap  \{ A \neq C\} \cap \{B = D\}\} = 0$.\\
Lastly, for case (iv), the coupling is 
\[\begin{cases}
Y' = D+1 \Longleftrightarrow X' = A+1\\
Y' = D-1 \Longleftrightarrow X' = A-1,\\
\end{cases}\]
giving that $\bP\{X' > Y' ~|~ E_1 \cap  \{ A = C\} \cap \{B = D\}\} = 0$.

Combining cases (i) through (iv) for $i=1$, we obtain that 
\begin{equation}\label{E_1'}\bP\{X' > Y' ~|~ E_1\} = 0.\end{equation}

Note that for $i\in \{2,3,4\}$, the event $E_i \cap \{A=C\} \cap \{B=D\}$ is empty, so it suffices to study sub-cases (i) through (iii). 

We next construct the coupling on $E_2$.
For case (i), we note that $E_2 \cap \{A = C\} \cap \{B \neq D\} = \{D < A = C \le B\}$. In this case the coupling is 
\[\begin{cases}
Y' = D \Longleftrightarrow X' = A+1\\
Y' = B \Longleftrightarrow X' = A-1,\\
\end{cases}\]
which gives $\bP\{X' > Y' ~|~ E_2 \cap \{A = C\} \cap \{B \neq D\} \} = 1/2$.\\
For case (ii) the coupling is 
\[\begin{cases}
Y' = D \Longleftrightarrow X' = C\\
Y' = B \Longleftrightarrow X' = A,\\
\end{cases}\]
which gives $\bP\{X' > Y' ~|~ E_2 \cap \{A \neq C\} \cap \{B \neq D\} \}= \frac{1}{2}.$\\
Next, for case (iii), $E_2 \cap \{A \neq C\} \cap \{B = D\} = \{A \le B=D < C\}$, and the coupling is 
\[\begin{cases}
Y' = D+1 \Longleftrightarrow X' = A\\
Y' = D-1 \Longleftrightarrow X' = C,\\
\end{cases}\]
which gives $\bP\{X' > Y' ~|~ E_2 \cap \{A \neq C\} \cap \{B = D\} \} = 1/2$.

Combining cases (i) through (iii) with $i=2$, we obtain that 
\begin{equation}\label{E_2'}\bP\{X' > Y' ~|~E_2\} = \frac{1}{2}.\end{equation}

We now construct the coupling on the event $E_3$. For case (i), $E_3 \cap \{A = C\} \cap \{B \neq D\} = \{B < A = C \le D\}$. The coupling is 
\[\begin{cases}
Y' = D \Longleftrightarrow X' = A-1\\
Y' = B \Longleftrightarrow X' = A+1,\\
\end{cases}\]
giving $\bP\{X' > Y' ~|~E_3 \cap \{A = C\} \cap \{B \neq D\}\} = 1/2$. \\
For case (ii), the coupling is
\[\begin{cases}Y' = D \Longleftrightarrow X' = C\\
Y' = B \Longleftrightarrow X' = A,\end{cases}\]
which gives $\bP\{X' > Y' ~|~ E_3 \cap \{A \neq C\} \cap \{B \neq D\} \}= \frac{1}{2}.$\\
Lastly, for case (iii), $E_3 \cap \{A \neq C\} \cap \{B = D\} = \{A > B = D \ge C\}$. The coupling is 
\[\begin{cases}
Y' = D+1 \Longleftrightarrow X' = C\\
Y' = D-1 \Longleftrightarrow X' = A,\\
\end{cases}\]
which gives  $\bP\{X' > Y' ~|~E_3 \cap \{A \neq C\} \cap \{B = D\}\} = 1/2$.

Combining cases (i) through (iii) with $i=3$, we obtain that 
\begin{equation}\label{E_3'}\bP\{X' > Y' ~|~E_3\} = \frac{1}{2}.\end{equation}

Finally, we construct the coupling on the event $E_4$ arbitrarily (for example, by making independent choices for the two processes). This gives 
\begin{equation}\label{E_4'}
\bP\{X' > Y' ~|~ E_4 \} \leq 1.\end{equation}

Since both $(A,B)$ and $(C,D)$ are distributed according to the coupling $(X,Y)$, we have $\bP(A>B) = \alpha = \bP(C>D)$. Since $(A,B)$ and $(C,D)$ are independent, we also have that
\[\bP\{E_2\} = \bP\{ A \leq B, C>D\} = \bP\{A>B, C\leq D\} = \bP\{E_3\} = \alpha(1-\alpha),\]
and $\bP\{E_4\} = \bP\{A>B, C>D\} = \alpha^2$.
Combining this with (\ref{E_1'}), (\ref{E_2'}), (\ref{E_3'}) and (\ref{E_4'}), and using the law of total probability, we obtain that 
\begin{align*}
    \bP\{X'>Y'\} & = \sum_{i=1}^4\bP\{X' > Y' ~|~ E_i \}\cdot \bP\{E_i\} \\
    & \leq \alpha(1-\alpha) + \alpha^2 = \alpha\, ,
\end{align*}
as required.
\end{proof}

\begin{proof}[Proof of Lemma~\ref{stochdom_talshrw}]
As noted above, the construction of a coupling with the claimed property is essentially identical to the construction from the proof of Lemma~\ref{stochdom}. To obtain it from  that construction, simply replace all instances of $A-1$ by $A$ (in cases $E_1$(i), $E_1$(iv), $E_2$(i) and $E_3$(i)) and all instances of $D-1$ by $D$ (in cases $E_1$(iii), $E_1$(iv), $E_2$(iii) and $E_3$(iii)). We leave the detailed verification to the reader.
\end{proof}

\renewcommand\refname{{\bf References}}
\bibliographystyle{plainnat}
\bibliography{hrw}

\begin{thebibliography}{11}
\providecommand{\natexlab}[1]{#1}
\providecommand{\url}[1]{\texttt{#1}}
\expandafter\ifx\csname urlstyle\endcsname\relax
  \providecommand{\doi}[1]{doi: #1}\else
  \providecommand{\doi}{doi: \begingroup \urlstyle{rm}\Url}\fi

\bibitem[Auffinger and Cable(2017)]{auffingercable2018}
Antonio Auffinger and Dylan Cable.
\newblock Pemantle's min-plus binary tree.
\newblock arXiv:1709.07849 [math.PR], September 2017.

\bibitem[Evje and Karlsen(2000)]{ek}
Steinar Evje and Kenneth~Hvistendahl Karlsen.
\newblock Monotone difference approximations of {BV} solutions to degenerate
  convection-diffusion equations.
\newblock \emph{SIAM J. Numer. Anal.}, 37\penalty0 (6):\penalty0 1838--1860,
  2000.
\newblock ISSN 0036-1429.
\newblock \doi{10.1137/S0036142998336138}.
\newblock URL \url{https://doi.org/10.1137/S0036142998336138}.

\bibitem[Godlewski and Raviart(1996)]{MR1410987}
Edwige Godlewski and Pierre-Arnaud Raviart.
\newblock \emph{Numerical approximation of hyperbolic systems of conservation
  laws}, volume 118 of \emph{Applied Mathematical Sciences}.
\newblock Springer-Verlag, New York, 1996.
\newblock ISBN 0-387-94529-6.
\newblock \doi{10.1007/978-1-4612-0713-9}.
\newblock URL \url{https://doi.org/10.1007/978-1-4612-0713-9}.

\bibitem[Hambly and Jordan(2004)]{MR2079916}
B.~M. Hambly and Jonathan Jordan.
\newblock A random hierarchical lattice: the series-parallel graph and its
  properties.
\newblock \emph{Adv. in Appl. Probab.}, 36\penalty0 (3):\penalty0 824--838,
  2004.
\newblock ISSN 0001-8678.
\newblock \doi{10.1239/aap/1093962236}.
\newblock URL \url{https://doi.org/10.1239/aap/1093962236}.

\bibitem[Harten et~al.(1976)Harten, Hyman, and Lax]{MR0413526}
A.~Harten, J.~M. Hyman, and P.~D. Lax.
\newblock On finite-difference approximations and entropy conditions for
  shocks.
\newblock \emph{Comm. Pure Appl. Math.}, 29\penalty0 (3):\penalty0 297--322,
  1976.
\newblock ISSN 0010-3640.
\newblock \doi{10.1002/cpa.3160290305}.
\newblock URL \url{https://doi.org/10.1002/cpa.3160290305}.
\newblock With an appendix by B. Keyfitz.

\bibitem[Kru\v{z}kov(1970)]{MR0267257}
S.~N. Kru\v{z}kov.
\newblock First order quasilinear equations with several independent variables.
\newblock \emph{Mat. Sb. (N.S.)}, 81 (123):\penalty0 228--255, 1970.

\bibitem[Lax(1973)]{MR0350216}
Peter~D. Lax.
\newblock \emph{Hyperbolic systems of conservation laws and the mathematical
  theory of shock waves}.
\newblock Society for Industrial and Applied Mathematics, Philadelphia, Pa.,
  1973.
\newblock Conference Board of the Mathematical Sciences Regional Conference
  Series in Applied Mathematics, No. 11.

\bibitem[Pemantle(2016)]{pemantlecourant}
Robin Pemantle.
\newblock Columbia-courant joint probability seminar series, 2016-10-14,
  October 2016.
\newblock URL 
  \href{http://web.archive.org/web/20161029170321/http:// www.math.nyu.edu:80/seminars/probability/ColumbiaCourant2016.html}{http://web.archive.org/web/20161029170321/http:// www.math.nyu.edu:80/seminars/probability/ColumbiaCourant2016.html}.

\bibitem[Vazquez(2006)]{Vazquez2006}
Juan~Luis Vazquez.
\newblock \emph{{The Porous Medium Equation: Mathemtical Theory}}.
\newblock Oxford Mathematical Monographs. Oxford University Press, 2006.
\newblock ISBN 9780198569039.
\newblock \doi{10.1093/acprof:oso/9780198569039.001.0001}.
\newblock URL
  \href{http://www.oxfordscholarship.com/view/10.1093/acprof:oso/9780198569039.001.0001/acprof-9780198569039}{http://www.oxfordscholarship.com/view/10.1093/acprof:oso/9780198569039.001.0001/acprof-9780198569039}.

\bibitem[Volpert(1967)]{MR0216338}
A.~I. Volpert.
\newblock Spaces {${\rm BV}$} and quasilinear equations.
\newblock \emph{Mat. Sb. (N.S.)}, 73 (115):\penalty0 255--302, 1967.

\bibitem[Volpert(2000)]{MR1785683}
A.~I. Volpert.
\newblock Generalized solutions of degenerate second-order quasilinear
  parabolic and elliptic equations.
\newblock \emph{Adv. Differential Equations}, 5\penalty0 (10-12):\penalty0
  1493--1518, 2000.
\newblock ISSN 1079-9389.

\end{thebibliography}

%
%
                 
\appendix

\end{document}